\documentclass[11pt,letterpaper]{article}
\usepackage{thmtools}
\usepackage{thm-restate}
\usepackage{bm}
\usepackage{bbm}
\usepackage{mathrsfs}           %
\usepackage{vmargin,fancyhdr}   %
\usepackage{tikz}
\usetikzlibrary{positioning,chains,fit,shapes,calc}

\usepackage{subcaption}
\usepackage{algorithm}
\usepackage{algorithmic}
\usepackage{enumerate}

\usepackage{amsmath,amssymb, amsthm}    %
\usepackage{verbatim}           %
\usepackage{xspace}             %
\usepackage{graphicx,float}     %
\usepackage{ifthen,calc}        %
\usepackage{textcomp}           %
\usepackage{fancybox}           %
\usepackage{hhline}             %
\usepackage{float}              %

\usepackage[vflt]{floatflt}     %
\usepackage[small,compact]{titlesec}
\usepackage{setspace}
\usepackage{color}
\definecolor{ForestGreen}{rgb}{0.1333,0.5451,0.1333}
\usepackage{thumbpdf}
\usepackage[letterpaper,
colorlinks,linkcolor=ForestGreen,citecolor=ForestGreen,
backref,
bookmarks,bookmarksopen,bookmarksnumbered]
{hyperref}

\setpapersize{USletter}
\setmarginsrb{1in}{.5in}        %
{1in}{.5in}        %
{.25in}{.25in}     %
{.25in}{.5in}      %
\newcommand{\showccc}[0]{0}
\newcommand{\ccc}[2][nothing]{%
	\ifthenelse{\showccc=0}{}{
		\ensuremath{^{\Lsh\Rsh}}\marginpar{\raggedright\tiny\textsf{%
				\ifthenelse{\equal{#1}{nothing}}{}{\textbf{#1}\\}#2}}}}
\newcounter{hours}\newcounter{minutes}
\newcommand{\hhmm}{%
	\setcounter{hours}{\time/60}%
	\setcounter{minutes}{\time-\value{hours}*60}%
	\ifthenelse{\value{hours}<10}{0}{}\thehours:%
	\ifthenelse{\value{minutes}<10}{0}{}\theminutes}
\ifthenelse{\showccc=0}{\rhead{}}{\rhead{\today \ [\hhmm]}}
\newtheorem{theorem}{Theorem}

\newtheorem{corollary}{Corollary}

\newtheorem{lemma}{Lemma}
\newtheorem{fact}{Fact}

\newtheorem{proposition}{Proposition}
\newtheorem{assumption}{Assumption}
\newtheorem{definition}{Definition}

\newtheorem*{rep@theorem}{\rep@title}
\newcommand{\newreptheorem}[2]{%
	\newenvironment{rep#1}[1]{%
		\def\rep@title{#2 \ref{##1}}%
		\begin{rep@theorem}}%
		{\end{rep@theorem}}}
\makeatother
\newreptheorem{theorem}{Theorem}

\usepackage{float}
\floatstyle{plain}\newfloat{myfig}{t}{figs}[section]
\floatname{myfig}{\textsc{Figure}}
\floatstyle{plain}\newfloat{myalg}{H}{algs}[section]
\floatname{myalg}{}
\definecolor{burntorange}{rgb}{0.8, 0.33, 0.0}

\newcommand{\defeq}{:=}
\newcommand{\norm}[1]{\left\lVert#1\right\rVert}
\newcommand{\inprod}[2]{\left\langle#1, #2\right\rangle}
\newcommand{\eps}{\epsilon}
\newcommand{\lam}{\lambda}
\newcommand{\argmax}{\textup{argmax}}
\newcommand{\argmin}{\textup{argmin}} 
\newcommand{\R}{\mathbb{R}}

\newcommand{\diag}[1]{\textbf{\textup{diag}}\left(#1\right)}
\newcommand{\half}{\frac{1}{2}}

\DeclareMathOperator*{\E}{\mathbb{E}}

\newcommand{\Nor}{\mathcal{N}}

\newcommand{\Tr}{\textup{Tr}}
\newcommand{\opt}{\textup{OPT}}

\newcommand{\ma}{\mathbf{A}}
\newcommand{\mb}{\mathbf{B}}
\newcommand{\mq}{\mathbf{Q}}
\newcommand{\mm}{\mathbf{M}}
\newcommand{\my}{\mathbf{Y}}
\newcommand{\ai}{\ma_{i:}}
\newcommand{\aj}{\ma_{:j}}
\newcommand{\id}{\mathbf{I}}
\newcommand{\tO}{\widetilde{O}}
\newcommand{\nnz}{\textup{nnz}}
\newcommand{\poly}{\textup{poly}}
\newcommand{\lmax}{\lambda_{\textup{max}}}
\newcommand{\lmin}{\lambda_{\textup{min}}}
\newcommand{\Par}[1]{\left(#1\right)}
\newcommand{\Brack}[1]{\left[#1\right]}
\newcommand{\Brace}[1]{\left\{#1\right\}}
\newcommand{\Abs}[1]{\left|#1\right|}
\newcommand{\msig}{\boldsymbol{\Sigma}}
\newcommand{\mzero}{\mathbf{0}}

\newcommand{\1}{\mathbbm{1}}

\newcommand{\ksl}{\kappa^\star_i}
\newcommand{\ksr}{\kappa^\star_o}
\newcommand{\ksk}{\kappa^\star_o}
\newcommand{\mn}{\mathbf{N}}
\newcommand{\Sym}{\mathbb{S}}
\newcommand{\PSD}{\Sym_{\succeq \mzero}}
\newcommand{\PD}{\Sym_{\succ \mzero}}
\newcommand{\mw}{\mathbf{W}}
\newcommand{\mk}{\mathbf{K}}
\newcommand{\mws}{\mw_\star}
\newcommand{\tw}{\tilde{w}}
\newcommand{\tx}{\tilde{x}}
\newcommand{\tv}{\tilde{v}}
\newcommand{\tmw}{\widetilde{\mw}}
\newcommand{\tmq}{\widetilde{\mq}}
\newcommand{\alg}{\mathcal{A}}
\newcommand{\algtest}{\alg_{\textup{test}}}
\newcommand{\algpack}{\alg_{\textup{pack}}}

\newcommand{\tmv}{\mathcal{T}_{\textup{mv}}}
\newcommand{\DecideMPC}{\mathsf{DecideStructuredMPC}}
\newcommand{\mg}{\mathbf{G}}
\newcommand{\mmu}{\mathbf{U}}
\newcommand{\bx}{\bar{x}}
\newcommand{\md}{\mathbf{D}}
\newcommand{\tma}{\widetilde{\ma}}

\newcommand{\ms}{\mathbf{S}}
\newcommand{\Power}{\mathsf{Power}}
\newcommand{\mpack}{\mathbf{P}}
\newcommand{\mcov}{\mathbf{C}}

\newcommand{\kscale}{\kappa_{\textup{scale}}}

\newcommand{\tsl}{\tau^\star_i}
\newcommand{\tsr}{\tau^\star_o}

\newcommand{\xtrue}{x_{\textup{true}}}
\newcommand{\MPC}{\mathsf{MPC}}
  \usepackage{nth}
  \usepackage{intcalc}

\usepackage{url}

\begin{document}

	\begin{titlepage}
		\def\thepage{}
		\thispagestyle{empty}
		\title{Fast and Near-Optimal Diagonal Preconditioning} 

		\date{}
		\author{
			Arun Jambulapati\thanks{Stanford University, {\tt jmblpati@stanford.edu}}
			\and
			Jerry Li\thanks{Microsoft Research, {\tt jerrl@microsoft.com}}
			\and
			Christopher Musco\thanks{New York University, {\tt cmusco@nyu.edu}}
			\and
			Aaron Sidford\thanks{Stanford University, {\tt sidford@stanford.edu}}
			\and
			Kevin Tian\thanks{Stanford University, {\tt kjtian@stanford.edu}}
		}
		\maketitle

\abstract{
The convergence rates of iterative methods for solving a linear system $\ma x = b$ typically depend on the condition number of the matrix $\ma$. Preconditioning is a common way of speeding up these methods by reducing that condition number in a computationally inexpensive way.
In this paper, we revisit the decades-old problem of how to best improve $\ma$'s condition number by left or right diagonal rescaling. We make progress on this problem in several directions. 

First, we provide new bounds for the classic heuristic of scaling $\ma$ by its diagonal values (a.k.a.\ Jacobi preconditioning). We prove that this approach reduces $\ma$'s condition number to within a quadratic factor of the best possible scaling. Second, we give a solver for structured mixed packing and covering semidefinite programs (MPC SDPs) which computes a constant-factor optimal scaling for $\ma$ in $\widetilde{O}\left(\nnz(\ma)\cdot \poly(\kappa^\star)\right)$ time; this matches the cost of solving the linear system after scaling up to a $\widetilde{O}(\poly(\kappa^\star))$ factor. Third, we demonstrate that a sufficiently general \emph{width-independent} MPC SDP solver would imply near-optimal runtimes for the scaling problems we consider, and natural variants concerned with measures of average conditioning.

Finally, we highlight connections of our preconditioning techniques to semi-random noise models, as well as applications in reducing risk in several statistical regression models.

} 		
	\end{titlepage}
	\pagenumbering{gobble}
	\setcounter{tocdepth}{2}
	{
		\hypersetup{linkcolor=black}
		\tableofcontents
	}
	\newpage
	\pagenumbering{arabic}

\section{Introduction}
\label{sec:intro}
Consider a linear system $\mk x = b$, where $\mk$ is a $d \times d$ matrix and $b$ is a $d$-dimensional vector. When $d$ is large, the most efficient algorithms for computing $x$ are usually iterative methods like the conjugate gradient method, gradient decent, or Chebyshev iteration \cite{Saad:2003}. These methods refine a solution over the course of multiple steps, and the number of steps required typically depends on an appropriate notion of \emph{condition number} of the matrix $\mk$. 

For example, let $\mk$ be positive definite with condition number $\kappa(\mk) \defeq \frac{\lmax(\mk)}{\lmin(\mk)}$, i.e.\ the ratio of its largest and smallest eigenvalues. The conjugate gradient method can compute an $\epsilon$-approximate solution to $x$ 
(for an appropriate definition of approximation) 
in $O(\sqrt{\kappa(\mk)} \log \frac 1 \eps))$ steps, where each step performs a single matrix-vector multiplication with $\mk$. While fast when $\mk$ is well-conditioned, computational costs of such methods grow when $\kappa(\mk)$ is large --- i.e.\ when $\mk$ is poorly conditioned.

Given the centrality of the condition number in the convergence of iterative methods, one common approach to improving runtimes is \emph{preconditioning} \cite{Greenbaum:1997}. While specifics come in many forms, the high-level goal of preconditioning is to find computationally inexpensive ways of reducing the condition number of a given linear system.\footnote{The minimum obtainable accuracy for solving a linear system on a finite precision computer is also limited
by the condition number, which offers a second motivation for preconditioning.} In this paper, we revisit one of the simplest, yet surprisingly effective, approaches for preconditioning: rescaling $\mk$'s rows or columns \cite{PiniGambolati:1990}.

\subsection{Problem statement}

We consider two versions of the rescaling problem, which have different properties and applications.

\medskip\noindent\textbf{Outer scaling.} We refer to the first rescaling problem considered as ``outer scaling.'' Let $\mk \in \R^{d \times d}$ be positive definite. The goal of outer scaling is to find a $d \times d$ positive diagonal matrix $\mw$ that minimizes or approximately minimizes\footnote{The symmetry of the rescaling $\mw^{\half} \mk \mw^{\half}$ is important, as it preserves the symmetry and positive definiteness of $\mk$, which is required for iterative solvers like the conjugate gradient method.}
\begin{align}
	\label{eq:outer}
	\kappa\Par{\mw^\half \mk \mw^\half}.
\end{align}
A solution to the linear system $\mk x = b$ can then be obtained by solving the better-conditioned system $\mw^{\half} \mk \mw^{\half}y = \mw^{\half}b$ and returning $x = \mw^{\half}y$. The problem of finding $\mw$ minimizing \eqref{eq:outer} has been studied since at least the 1950s \cite{ForsytheStraus:1955,Sluis:1969}. While the optimal diagonal $\mw$ can be found via semidefinite programming (SDP), the computational overhead of state-of-the-art general SDP solvers \cite{JiangKLP020, HuangJST21} would outweigh any benefits from reducing the condition number of $\mk$, since all known exact solvers at least require solving one linear system. Our goal is to develop much faster algorithms for computing an $\mw$ approximately minimizing \eqref{eq:outer}. 

We also consider the outer scaling problem in the common setting when we have access to a factorization of a positive definite matrix $\mk$ --- i.e.\ we have some matrix $\ma\in \R^{n\times d}$ where $\ma^\top \ma = \mk$. 
This case arises e.g.\ when computing $\min_x\|\ma x - b\|_2$, i.e. the solution of a least squares regression problem. The goal of the outer rescaling problem in this case is still to minimize $\kappa(\mw^\half \ma^\top \ma \mw^\half)$. 

\medskip\noindent\textbf{Inner scaling.} The second problem we consider is an ``inner scaling'' problem. 
Given a full rank $n \times d$ matrix $\ma$ with $n \ge d$, the goal of inner scaling is to find an $n\times n$ nonnegative diagonal matrix $\mw$ that minimizes or approximately minimizes
\begin{align}
	\label{eq:inner}
	\kappa\Par{\ma^\top \mw \ma}.
\end{align}
In contrast to the outer scaling problem for $\ma^\top \ma$, where the goal is to rescale columns of $\ma$ by $\mw^\half$, the inner scaling problem asks to rescale $\ma$'s rows. To the best of our knowledge, this problem has not been studied previously. It is not strictly a preconditioning problem, as a good inner scaling does not allow for a faster solution to a given least squares regression problem $\min_x \|\ma x - b\|_2$. Instead, it allows for a faster solution to the weighted least squares problem $\min_x \|\mw^{\half}\ma x - \mw^{\half}b\|_2$. 

The inner scaling problem has an interesting connection to \emph{semi-random} noise models for 
least-squares regression, which we highlight in Section~\ref{sec:apps}. As a motivating example, consider the case when there is a hidden parameter vector $\xtrue \in \R^d$ that we want to recover, and we have a ``good'' set of consistent observations $\ma_g \xtrue = b_g$, in the sense that $\kappa(\ma_g^\top \ma_g)$ is small. Here, we can think of $\ma_g$ as being drawn from a well-conditioned distribution. Now, suppose an adversary gives us a superset of these observations $(\ma, b)$ such that $\ma \xtrue = b$, and $\ma_g$ are an (unknown) subset of rows of $\ma$, but $\kappa(\ma^\top\ma) \gg \kappa(\ma_g^\top \ma_g)$. Perhaps counterintuitively, by giving us additional consistent data, the adversary can arbitrarily hinder the cost of iterative methods for finding $\xtrue$. Our inner scaling methods can be viewed as a way of ``robustifying'' linear system solvers to such noise models (by finding $\mw$ at least as good as the indicator of the rows of $\ma_g$, which are not known a priori). We also demonstrate additional applications in more general statistical regression settings.

\subsection{Our contributions}

We provide several new results on efficiently computing near-optimal diagonal scalings for approximately minimizing \eqref{eq:outer} and \eqref{eq:inner}. We denote the optimal solutions for these problems as follows.

\begin{definition}[Optimal scaled condition numbers]\label{def:optimal_condition_number}
	For positive definite $\mk\in \R^{d\times d}$ and full-rank $\ma\in \R^{n\times d}$ with $n \ge d$, we define\footnote{Without loss of generality, the minimization is over a compact set (and hence is a minimum and not an infimum). To see this, we can assume by scale invariance that $\mw \succeq \id$, and this places a finite upper bound on the range of $\mw$ as well (beyond which $\mw^\half \mk \mw^\half$ would have larger condition number than $\mk$).}
	\begin{align*}
		\ksk(\mk) \defeq \min_{\textup{diagonal } \mw \succeq \mzero}  \kappa\Par{\mw^\half \mk \mw^\half}\text{    and    } \ksl(\ma) \defeq \min_{\textup{diagonal } \mw \succeq \mzero}  \kappa\Par{\ma^\top \mw \ma}.
	\end{align*}
\end{definition}

Our first result concerns a classic heuristic for solving the outer scaling problem known as \emph{Jacobi preconditioning} --- simply choose $\mw$ to be the inverse-diagonal of $\mk$, so $\mw^\half \mk \mw^\half$ has a constant diagonal. Equivalently, if we have a factorization $\ma^\top\ma = \mk$, we choose $\mw_{jj} = \|\aj\|_2^{-2}$, where $\aj$ is the $j^\text{th}$ column of $\ma$. 
In Section \ref{sec:diagonalone}, we prove that this choice of $\mw$ always satisfies:
\begin{align*}
	\kappa\Par{\mw^{\half} \mk \mw^{\half}} \leq \kappa_o^*\left(\mk\right)^2.
\end{align*}
This bound adds a \emph{dimension-independent} result to the existing literature on the Jacobi preconditioner, which was previously known to achieve an approximation of $\kappa(\mw^\half \mk \mw^\half) \leq d\cdot \ksk(\mk)$ \cite{Sluis:1969}. We also prove a matching lower bound, showing that the quadratic approximation factor is tight: there is a matrix $\mk$ for which the Jacobi preconditioner $\mw$ gives $\kappa(\mw^\half \mk \mw^\half) = \Omega(\ksk(\mk)^2)$. 
Our upper bound and lower bound are given as Propositions~\ref{prop:diagonesub} and~\ref{prop:diagoneslb} respectively.

Our second set of results is more technically involved. We give algorithms for solving the outer and inner scaling problems up to a constant factor. In Section \ref{sec:slowoptimal} we prove the following.
\begin{theorem}[informal, see Theorems~\ref{thm:leftslow} and~\ref{thm:rightslow}]\label{thm:intro}
	Let $\ma \in \R^{n \times d}$ be full-rank with $n \geq d$. There is a randomized algorithm that with high probability computes a diagonal $\mw\in \R^{d\times d}$ satisfying $\kappa(\mw^\half \ma^\top  \ma\mw^\half ) \leq 2 \cdot \ksk(\ma^\top \ma)$  in time:
	\begin{align*}
		\tO\left(\nnz(\ma)\cdot \ksk(\ma^\top \ma)^{1.5}\right).
	\end{align*}
A variant of the same algorithm computes a diagonal $\mw\in \R^{n\times n}$ satisfying $\kappa(\ma^\top \mw \ma) \leq 2  \cdot \ksl(\ma)$  in time $\tO(\nnz(\ma)\cdot \ksl(\ma)^{1.5})$. 
\end{theorem}

In the above and throughout the paper, $\nnz(\ma)$ denotes the number of nonzero entries in $\ma$, and $\tO$ hides polylogarithmic factors in its argument
and the algorithm's failure probability: detailed runtimes are given in Theorems \ref{thm:leftslow} and  \ref{thm:rightslow}. The stated approximation factor of 2 is arbitrary --- with larger constants in the runtimes, we can achieve a $(1+\epsilon)$-approximation for any constant $\epsilon > 0$. However, we note that a $(1+\epsilon)$-approximation offers little advantage over a 2 factor approximation for most applications, where the goal is to solve a linear system in $\mw^\half\ma^\top\ma\mw^\half$ or $\ma^\top\mw\ma$.
\smallskip

Theorem \ref{thm:intro} is achieved by developing a highly-optimized solver for a specific mixed packing and covering (MPC) SDP which is used to find $\mw$. Importantly, the runtime of this solver nearly-matches the cost of solving a linear system in $\mw^{\half}\ma^\top \ma\mw^{\half}$ or $\ma^\top \mw\ma$ after optimally scaling, \emph{if the best scaling was known a priori}: $\tO(\nnz(\ma)\cdot \sqrt{\ksk(\ma^\top\ma)})$ and $\tO(\nnz(\ma)\cdot \sqrt{\ksl(\ma)})$, respectively. The runtime does not depend on the condition number of the original (non-rescaled) matrix $\ma$.

In comparison to our new result on the Jacobi preconditioner, Theorem \ref{thm:intro} does not immediately offer an advantage for the outer scaling problem, where the ultimate goal is to solve a linear system involving $\ma^\top \ma$. 
In particular, while obtaining better condition number, the computational overhead of computing $\mw$ means that we get an algorithm for finding $x$ with runtime depending on $\ksl(\ma)^{1.5}$ instead of $\sqrt{\ksl(\ma)^2} = \ksl(\ma)$. That said, when repeatedly solving multiple linear systems in the same matrix, Theorem \ref{thm:intro} could easily offer an advantage over Jacobi preconditioning. For the inner scaling problem, where Jacobi preconditioning does not apply, Theorem \ref{thm:intro} yields the first $\tO\left(\nnz(\ma)\cdot \poly(\ksl(\ma)\right)$ algorithm for solving an overdetermined consistent linear system in $\ma$. In Section~\ref{sec:apps}, we show applications of the inner scaling problem in several semi-random statistical models. For example, our methods enable efficient overdetermined linear system solving in a semi-random model where we receive a \emph{superset} of a well-conditioned set of measurements, but the superset can be arbitrarily poorly conditioned.

Moreoever, we provide an interesting direction towards improved runtime bounds and a broader connection to the literature on positive semidefinite programs. In particular, we show a sufficiently general \emph{width-independent} algorithm for approximately solving MPC SDP problems of the form
\[\text{does there exist } w \in \R^n_{\ge 0} \text{ such that } \lmax\Par{\sum_{i =1}^n w_i \mpack_i} \le \lmin\Par{\sum_{i =1}^n w_i \mcov_i}?\]
would imply Theorem \ref{thm:intro}'s runtime can be improved to have a square root dependence on the optimal condition number. Such a bound would allow for optimal preconditioning with essentially no computational overhead in comparison to the optimally-reweighted linear system solve. In Section \ref{sec:conjoptimal}, we state explicit requirements of such a (conjectured) improved solver which would imply such bounds.\footnote{An earlier version of this paper, including of a subset of the results described in the present work, was based on the MPC SDP solver of \cite{JambulapatiLLPT20}, which claimed to achieve the requirements stated in Section~\ref{sec:conjoptimal} and thus would obtain near-optimal rates for our scaling problems. However, an error was later discovered in \cite{JambulapatiLLPT20}, as recorded in the newest arXiv version of that work. For completeness and to highlight the connection between these problems, we include this set of results as a reduction to a conjectured subroutine in Section~\ref{sec:conjoptimal}.} 
We also demonstrate how such a general MPC SDP solver implies near-optimal rates for rescalings achieving \emph{average notions} of condition number, a common parameterization for runtimes of recent stochastic linear system solvers \cite{StrohmerV06, LeeS13, Johnson013, DefazioBL14, ZhuQRY16, Allen-Zhu17, AgarwalKKLNS20}.

Finally, while the statement of Theorem \ref{thm:intro} assumes access to a factorization $\ma^\top \ma$ of a matrix $\mk$ we wish to outer scale (i.e.\ we know $\ma$ such that $\mk = \ma^\top \ma$), we show how a modified implementation of our SDP solver, combined with a natural homotopy method and square root approximations, obtains a similar result when we only have black-box matrix-vector access to positive definite $\mk$ (without an explicit factorization). The following result, proven in Section \ref{sec:kernel}, addresses applications to e.g.\ solving systems involving large kernel matrices or implicit PD matrices arising in scientific computing applications.
\begin{theorem}[informal, see Theorem~\ref{thm:slowsqrt}]\label{thm:intro_kernel}
	Let $\mk \in \R^{d \times d}$ be positive definite. There is a randomized algorithm that with high probability computes a diagonal $\mw\in \R^{d\times d}$ satisfying $\kappa(\mw^{\half}\mk\mw^{\half}) \leq 2 \cdot \ksk(\mk)$ using $\tO(\ksk(\mk)^{1.5})$ matrix-vector products with $\mk$.
\end{theorem}

\subsection{Our techniques}
Besides our bounds for the Jacobi preconditioner, which are self-contained and discussed in Section~\ref{sec:diagonalone},  our main results are based on developing an efficient specialized SDP solver. 
We first consider the inner scaling problem and then discuss how a solution for this problem can be generalized to the outer scaling problem.
Our starting point is the observation that, to find a near optimal inner scaling, it suffices to solve the following problem with $\kappa$ chosen to equal $\ksl(\ma)$:
\begin{equation}\label{eq:decisionmpc}\text{find } w \in \R^n_{\ge 0} \text{ such that } \lmax\Par{\sum_{i =1}^n w_i a_ia_i^\top} \le 2\kappa\cdot  \lmin\Par{\sum_{i =1}^n  w_i a_ia_i^\top}.\end{equation}
We observe that \eqref{eq:decisionmpc} is a structured case of a mixed packing and covering SDP (as defined in \cite{JambulapatiLLPT20}), where the packing matrices and covering matrices are identical (variants of which were studied in \cite{Lee017, ChengG18, JambulapatiSS18}). By applying an oracle-based matrix multiplicative weights method (akin to those developed in \cite{AroraHK12}), as well as known efficient approximation algorithms for pure packing SDPs, we show that we can solve \eqref{eq:decisionmpc} in $O(\kappa \log d)$ iterations, as long as there exists a $w$ for which $\lmax\Par{\sum_{i } w_i a_ia_i^\top} \le \kappa\cdot  \lmin\Par{\sum_{i} w_i a_ia_i^\top}$. 
By searching over geometrically increasing values of $\kappa$, our method finds a near-optimal scaling and never requires more than $O(\ksl(\ma) \log d)$ iterations, even when $\ksl(\ma)$ is not known in advance.

It remains to show how to implement each iteration of the solver for  \eqref{eq:decisionmpc} efficiently. In short, each iteration is bottlenecked by computations involving matrix exponentials of a current reweighting $\ma^\top \mw \ma$ maintained by the algorithm. We first show that our solver is robust to approximations of these computations. Then,  by exploiting boundedness properties of $\ma^\top \mw \ma$, and leveraging low-degree polynomial approximations to the exponential, we show that each iteration is implementable in $\tO(\nnz(\ma) \cdot \sqrt{\ksl(\ma)})$ time due to the degrees of various polynomial approximations used in our method, concluding our proof for the inner scaling algorithm. 

For the outer scaling problem, we first consider the case where we are given a factorization $\mk = \ma^\top\ma$. Our first observation is that the eigenvalues of $\ma \mw \ma^\top \in \R^{n \times n}$ are the same as those of $\mw^\half \ma^\top \ma \mw^\half \in \R^{d \times d}$, except that there is an additional eigenvalue $0$ with multiplicity $n - d$. So, we can almost apply our inner scaling algorithm directly to $\ma^\top$, but need a way of dealing with these $0$ eigenvalues. We do so by extending $\ma^\top$ to an $n \times n$ full-rank matrix $\mb$ by completing its row span with an ``imagined'' basis. We implicitly update the weights on the imagined basis without actually computing it, and thus obtain a similar runtime for approximating $\ksr(\ma^\top\ma)$.

Finally, we consider the case when we do not have a factorization of $\mk$ (Theorem~\ref{thm:intro_kernel}). We show that our solver can be implemented given only matrix-vector multiplication access to the positive definite square root of $\mk$, $\ma = \mk^\half$. The required multiplications can be computed approximately via a matrix polynomial method. However, accurate polynomial approximations to the square root have degrees depending on $\kappa(\mk)$, which we were aiming to improve. We implement a homotopy method (reminiscent of techniques used by \cite{LiMP13, KapralovLMMS14, BubeckCLL18, AdilKPS19}) to break this chicken-and-egg problem, iteratively computing reweightings to $\mk + \lam \id$ for slowly-decreasing $\lam$ in logarithmically many phases. We remark that up to logarithmic factors, this result subsumes the result discussed in the previous paragraph.

\subsection{Related work}

While our inner scaling problem is new, the outer scaling problem --- i.e.\ that of computing an optimal  diagonal preconditioner --- has been widely studied. The most well-known existing results are those of van der Sluis \cite{Sluis:1969,GreenbaumRodrigue:1989}, who proved that the Jacobi preconditioner gives an $m$-factor approximation for the outer scaling problem where $m \leq d$ is the maximum number of non-zeros in any row of $\mk$. We review and slightly strengthen this result in Section \ref{sec:diagonalone}. It is also possible to prove that for certain limited classes of matrices, Jacobi preconditioning is actually optimal, and understanding these classes has been the subject of past work \cite{LivneGolub:2004,ForsytheStraus:1955}.

Other heuristics without provable approximation guarantees have also been studied. For example, a common approach is to choose $\mw$ to minimize $\|\id - \mw^\half\mk\mw^\half\|_{\textup{F}}$, where $\id$ is the $d\times d$ identity \cite{GroteHuckle:1997,BenziTuma:1999}. For solving the problem optimally, it has been noted that it suffices to solve a specific SDP \cite{QuYeZhou:2020}, but efficient algorithms (i.e.\ preconditioning algorithms with runtimes comparable to the cost of the linear system solve after preconditioning) have not been proposed.

\medskip\noindent\textbf{{Relationship to other problems.}} The problems studied in this paper bear superficial resemblance to the \emph{matrix scaling problem}, which has received recent attention in theoretical computer science and has a long history in scientific computing \cite{Allen-ZhuLiOliveira:2017,CohenMadryTsipras:2017,GargOliveira:2018}. The goal in matrix scaling is to find a diagonal scaling that equalizes the row and column norms of a matrix $\mk$. Matrix scaling has been used as a heuristic to improve the numerical stability of eigenvalue algorithms, to reduce the need for pivoting in direct solvers, and to improve $\mk$'s condition number \cite{KnightRuizUcar:2014}. However, the problem targets a different objective, so scaling algorithms do not yield provable guarantees on $\kappa ( \mw^{\half}\mk\mw^{\half})$ more generally.

Diagonal preconditioning is also related at high-level to \emph{adaptive preconditioning} methods like Adam \cite{KingmaBa:2015}, AdaGrad \cite{DuchiHazanSinger:2011}, and RMSProp \cite{HintonSrivastavaSwersky:2014} which are popular in machine learning applications. These methods rescale iterative steps by a diagonal matrix at each step of a stochastic gradient iteration, which would be similar to running gradient descent on a diagonally scaled $\mk$. A key difference however is that the preconditioner changes at each iteration, and these methods are often applied to non-quadratic, or even non-convex problems. 

\medskip\noindent\textbf{{Positive semidefinite programming.}} Our main algorithmic framework follows from developing a solver for a specialized mixed packing and covering SDP of the form described in \eqref{eq:decisionmpc}. A similar formulation was studied in \cite{JambulapatiSS18} as their Problem 3.1; this paper provides an different algorithm ours (in Section~\ref{sec:slowoptimal}) which is based in part on work of \cite{Lee017} which achieves a runtime polynomial in $\kappa^\star$ and the inverse multiplicative accuracy. 
We provide a tighter analysis which gives a refined dependence on $\kappa^\star$ compared to \cite{Lee017, JambulapatiSS18}.\footnote{We remark that techniques of \cite{Lee017} were also applied in \cite{ChengG18}, but no new runtimes appear to be claimed.}
We believe our techniques have more general applications, e.g.\ to Problem 3.1 of \cite{JambulapatiSS18}, and defer exploring this direction to future work. We give a more thorough comparison of problem formulations and algorithmic differences, in Appendix~\ref{app:mpcdiscuss}.

An important unanswered problem in the field of positive semidefinite programming is the development of an efficient \emph{width-independent} solver for general MPC SDPs. Such a solver would subsume the specific formulation in our work and \cite{JambulapatiSS18} (where the packing and covering matrices are multiples of each other). Such solvers are only currently known in the case of pure packing SDPs \cite{PengTZ16, Allen-ZhuLO16, Jambulapati0T20}. We include Section~\ref{sec:conjoptimal} as an interesting application such a solver, leaving open the possibility of constructions for diagonal preconditioners with near-optimal runtimes.

\medskip\noindent\textbf{{Semi-random noise models.}} The semi-random noise model we introduce in Section~\ref{sec:apps} for linear system solving follows a line of work originating in \cite{BlumS95} for graph coloring. A semi-random model consists of an (unknown) planted instance which a classical algorithm performs well against, augmented by additional information given by a ``monotone'' or ``helpful'' adversary masking the planted instance. Conceptually, when an algorithm fails given this ``helpful'' information, the algorithm may have overfit to its problem specification. This model has been studied in various statistical settings \cite{Jerrum92, FeigeK00, FeigeK01, MoitraPW16, MakarychevMV12}. Of particular relevance to our work, which studies robustness to semi-random noise in the context of fast algorithms (as opposed to the distinction between polynomial-time algorithms and computational intractability) is \cite{ChengG18}, which developed an algorithm for matrix completion under semi-random noise extending work of \cite{Lee017}.  	%

\section{Preliminaries}
\label{sec:prelims}

\paragraph{General notation.} We let $[n] \defeq \{1,2,\cdots,n\}$. Applied to a vector, $\norm{\cdot}_p$ is the $\ell_p$ norm. Applied to a matrix, $\norm{\cdot}_2$ is overloaded to denote the $\ell_2$ operator norm. $\Nor(\mu, \msig)$ denotes the multivariate Gaussian with specified mean and covariance. $\Delta^n$ is the simplex in $n$ dimensions (the subset of $\R^n_{\ge 0}$ with unit $\ell_1$ norm). We use $\tO$ to hide polylogarithmic factors in problem conditioning, dimensions, the target accuracy, and the failure probability. We say $\alpha\in \R$ is an $(\eps, \delta)$-approximation to $\beta \in \R$ if $\alpha = (1 + \eps')\beta + \delta'$, for $|\eps'| \le \eps$, $|\delta'| \le \delta$. An $(\eps, 0)$-approximation is an ``$\eps$-multiplicative approximation'' and a $(0, \delta)$-approximation is a ``$\delta$-additive approximation''. We let $\Nor(\mu, \msig)$ denote the multivariate Gaussian distribution of specified mean and covariance.

\paragraph{Matrices.} Throughout, matrices are denoted in boldface. We use $\nnz(\ma)$ to denote the number of nonzero entries of a matrix $\ma$. The set of $d \times d$ symmetric matrices is denoted $\Sym^d$, and the positive semidefinite and definite cones are $\PSD^d$ and $\PD^d$ respectively. For $\ma\in \Sym^d$, let $\lmax(\ma)$, $\lmin(\ma)$, and $\Tr(\ma)$ denote the largest magnitude eigenvalue, smallest eigenvalue, and trace. For $\ma \in \PD^d$, let $\kappa(\mm) \defeq \frac{\lmax(\mm)}{\lmin(\mm)}$ denote the condition number.
The inner product between matrices $\mm, \mn \in \Sym^d$ is the trace product, $\inprod{\mm}{\mn} \defeq \Tr(\mm\mn) = \sum_{i, j \in [d]} \mm_{ij} \mn_{ij}$. We use the Loewner order on $\Sym^d$: $\mm \preceq \mn$ if and only if $\mn - \mm \in \PSD^d$. $\id$ is the identity of appropriate dimension when clear. $\diag{w}$ for $w \in \R^n$ is the diagonal matrix with diagonal entries $w$.
For $\mm \in \PD^d$, $\norm{v}_{\mm} \defeq \sqrt{v^\top\mm v}$. For $\mm \in \Sym^d$ with eigendecomposition $\mathbf{V}^\top\boldsymbol{\Lambda}\mathbf{V}$, $\exp(\mm) \defeq \mathbf{V}^\top\exp(\boldsymbol{\Lambda})\mathbf{V}$, where $\exp(\boldsymbol{\Lambda})$ is applies entrywise to the diagonal. Similarly for $\mm = \mathbf{V}^\top\boldsymbol{\Lambda}\mathbf{V} \in \PSD^d$, $\mm^\half \defeq \mathbf{V}^\top\boldsymbol{\Lambda}^\half\mathbf{V}$. We denote the rows and columns of $\ma \in \R^{n \times d}$ by $\ai$ for $i \in [n]$ and $\aj$ for $j \in [d]$ respectively. 
Finally, $\tmv(\mm)$ denotes the time it takes to multiply a vector $v$ by $\mm$. We similarly denote the total cost of vector multiplication through a set $\{\mm_i\}_{i \in [n]}$ by $\tmv(\{\mm_i\}_{i \in [n]})$. We assume that $\tmv(\mm) = \Omega(d)$ for any $d \times d$ matrix, as that time is generally required to write the output.

\paragraph{Technical facts.} To speed up the iterations in our algorithm, we leverage several standard results on random projections and polynomial approximations to $e^x$ and $\sqrt{x}$.

\begin{fact}[Johnson-Lindenstrauss \cite{DasguptaG03}]\label{fact:jl}
	For $0 \leq \eps \leq 1$, let $k = \Theta\left(\frac{1}{\eps^2}\log \frac d \delta\right)$ for an appropriate constant. For $\mq \in \R^{k \times d}$ with independent uniformly random unit vector rows in $\R^d$ scaled down by $\frac 1 {\sqrt k}$, with probability $\ge 1 - \delta$ for any fixed $v \in \R^d$,
	\[(1 - \eps) \norm{\mq v}_2^2 \le \norm{v}_2^2 \le (1 + \eps)\norm{\mq v}_2^2. \]
\end{fact}

\begin{fact}[Polynomial approximation of $\exp$ \cite{SachdevaV14}, Theorem 4.1]\label{fact:polyexp}
	Let $\mm \in \PSD^d$ have $\mm \preceq R\id$. Then for any $\delta > 0$, there is an explicit polynomial $p$ of degree $O(\sqrt{R \log \frac 1 \delta + \log^2 \frac 1 \delta} )$ with
	\[\exp(-\mm) - \delta \id \preceq p(\mm) \preceq \exp(-\mm) + \delta \id.\]
\end{fact}

\begin{restatable}[Polynomial approximation of $\sqrt{\cdot}$]{fact}{restatepolysqrt}\label{fact:poly_approx_sqrt}
Let $\mm \in \PD^d$ have $\mu \id \preceq \mm \preceq \kappa \mu \id$ where $\mu$ is known. Then for any $\delta \in (0, 1)$, there is an explicit polynomial $p$ of degree $O(\sqrt \kappa \log \frac \kappa \delta)$ with
\[(1 - \delta)\mm^\half \preceq p(\mm) \preceq (1 + \delta)\mm^\half.\]
\end{restatable}
We give a proof of Fact~\ref{fact:poly_approx_sqrt} in Appendix~\ref{app:prelims} for completeness, using a result on polynomial approximations of analytic functions. 	%

\section{Normalizing the diagonal}
\label{sec:diagonalone}

In this section, we analyze a popular heuristic for computing diagonal preconditioners. Given a positive definite matrix $\mk \in \PD^d$, consider applying the outer scaling
\begin{equation}\label{eq:diagones}\mw^{\half} \mk \mw^{\half},\text{ where } \mw = \diag{w} \text{ and } w_i \defeq \mk_{ii}^{-1} \text{ for all } i \in [d].\end{equation}
In other words, the result of this scaling is to simply normalize the diagonal of $\mk$ to be all ones; we remark $\mw$ has strictly positive diagonal entries, else $\mk$ is not positive definite. Also called the Jacobi preconditioner, a result of Van de Sluis \cite{GreenbaumRodrigue:1989,Sluis:1969} proves that \emph{for any matrix} this scaling leads to a condition number that is within an $m$ factor of optimal, where $m \leq d$ is the maximum number of non-zeros in any row of $\mk$.
For completeness, we state a generalization of Van de Sluis's result below. We also require a simple fact; both are proved in Appendix~\ref{app:prelims}.

\begin{restatable}{fact}{restateprodkappa}\label{fact:prodkappa}
For any $\ma, \mb \in \PD^d$, $\kappa(\ma^\half\mb\ma^\half) \le \kappa(\ma) \kappa(\mb)$.
\end{restatable}

\begin{restatable}{proposition}{restateddd}\label{prop:diagdimensiondependence}
	Let $\mw$ be defined as in \eqref{eq:diagones} and let $m$ denote the maximum number of non-zero's in any row of $\mk$. Then,
	\[\kappa\Par{\mw^{\half} \mk \mw^{\half}} \le \min\left(m,\sqrt{\nnz(\mk)}\right)\cdot \ksk\Par{\mk}.\]
\end{restatable}
Note that $m$ and $\sqrt{\nnz(\mk)}$ are both $\leq d$, so it follows that 
	$\kappa(\mw^{\half} \mk \mw^{\half}) \leq {d}\cdot \ksk\Par{\mk}$. While the approximation factor in  Proposition \ref{prop:diagdimensiondependence} depends on the dimension or sparsity of $\mk$, we show that a similar analysis actually yields a \emph{dimension-independent} approximation. Specifically, the Jacobi preconditioner always obtains condition number no worse than the optimal squared. To the best of our knowledge, this simple but powerful bound has been observed in prior work. 

\begin{proposition}\label{prop:diagonesub}
Let $\mw$ be defined as in \eqref{eq:diagones}. Then,
\[\kappa\Par{\mw^{\half} \mk \mw^{\half}} \le \Par{\ksk\Par{\mk}}^2.\]
\end{proposition}

\begin{proof}
Let $\mws$ attain the minimum in the definition of $\ksk$ (Definition~\ref{def:optimal_condition_number}), i.e.\ $\kappa(\mk_\star) = \ksk(\mk)$ for $\mk_\star \defeq \mws^{\half} \mk \mws^{\half}$. Note that since $[\mw^{\half} \mk \mw^{\half}]_{ii} = 1$ by definition of $\mw$ it follows that for all $i$
\[
[\mws \mw^{-1}]_{ii}
= [\mws \mw^{-1}]_{ii} \cdot [\mw^{\half} \mk \mw^{\half}]_{ii} =
 [\mk_\star]_{ii}  = e_i^\top \mk_\star e_i \in [\lambda_{\min}(\mk_\star), \lam_{\max}(\mk_\star)] 
\]
where the last step used that $ \lambda_{\min}(\mk_\star) \id \preceq \mk_\star \preceq \lambda_{\max}(\mk_\star) \id$. Consequently, for  $\tmw \defeq \mws^{-1} \mw$ it follows that $\kappa(\tmw) = \kappa(\tmw^{-1}) \leq \lambda_{\max}(\mk_\star) / \lam_{\min}(\mk_\star) = \kappa(\mk_\star)$. The result follows from Fact~\ref{fact:prodkappa} as
\begin{align*}
	\kappa\Par{\mw^\half \mk \mw^\half} &= \kappa\Par{\tmw^\half \mk_\star\tmw^\half} \le \kappa\Par{\tmw}\kappa\Par{\mk_\star} \le \Par{\ksk}^2. \qedhere
\end{align*}
\end{proof}

Next, we demonstrate that Proposition~\ref{prop:diagonesub} is essentially tight by exhibiting a family of matrices which attain the bound of Proposition~\ref{prop:diagonesub} up to a constant factor. At a high level, our strategy is to create two blocks where the ``scales'' of the diagonal normalizing rescaling are at odds, whereas a simple rescaling of one of the blocks would result in a quadratic savings in conditioning.

\begin{proposition}\label{prop:diagoneslb}
Consider a $2d \times 2d$ matrix $\mm$ such that
\[\mk = \begin{pmatrix} \ma & \mzero \\ \mzero & \mb \end{pmatrix},\; \ma = \sqrt{d} \id + \1\1^\top,\; \mb = \id - \frac 1 {\sqrt d + d}\1\1^\top, \]
where $\ma$ and $\mb$ are $d \times d$. Then, defining $\mw$ as in \eqref{eq:diagones}, 
\[\kappa\Par{\mw^{\half} \mk \mw^{\half}} = \Theta(d),\; \ksk\Par{\mk} = \Theta\Par{\sqrt d}.\]
\end{proposition}
\begin{proof}
Because $\mw^{\half}\mk\mw^{\half}$ is blockwise separable, to understand its eigenvalue distribution it suffices to understand the eigenvalues of the two blocks. First, the upper-left block (the rescaling of the matrix $\ma$) is multiplied by $\frac{1}{\sqrt{d} + 1}$. It is straightforward to see that the resulting eigenvalues are
\[\frac{\sqrt d}{\sqrt d + 1} \text{ with multiplicity } d - 1,\; \sqrt d \text{ with multiplicity } 1.\]
Similarly, the bottom-right block is multiplied by $\frac{d + \sqrt d}{d + \sqrt d - 1}$, and hence its rescaled eigenvalues are
\[\frac{d + \sqrt d}{d + \sqrt d - 1} \text{ with multiplicity } d - 1,\; \frac{\sqrt d}{d + \sqrt d - 1} \text{ with multiplicity } 1.\]
Hence, the condition number of $\mw^\half \mk \mw^\half$ is $d + \sqrt d - 1 = \Theta(d)$. However, had we rescaled the top-left block to be a $\sqrt{d}$ factor smaller, it is straightforward to see the resulting condition number is $O(\sqrt{d})$. On the other hand, since the condition number of $\mk$ is $O(d)$,  Proposition~\ref{prop:diagonesub} shows that the optimal condition number $\ksk(\mk)$ is $\Omega(\sqrt d)$, and combining yields the claim. We remark that as $d \to \infty$, the constants in the upper and lower bounds agree up to a low-order term.
\end{proof} 	%

\section{Constant-factor optimal inner and outer scalings}
\label{sec:slowoptimal}

In this section, we give near-linear time constructions of inner and outer scalings which attain the optimal reweighted condition number up to a constant factor. In the latter case, we will assume knowledge of a factorization $\mk = \ma^\top \ma$ in this section. As a starting point to our development, we develop a custom MPC SDP solver for structured instances (namely, when the covering matrices are multiples of the packing matrices) in Section~\ref{ssec:mpcslow}. Using this primitive, we then handle computation of optimal inner rescalings in Section~\ref{ssec:kslslow} and outer rescalings in Section~\ref{ssec:ksrslow}.

\subsection{Reducing structured mixed packing-covering to pure packing}\label{ssec:mpcslow}

We first develop a general solver for the following structured ``mixed packing-covering'' decision problem: given a set of matrices $\{\ma_i\}_{i \in [n]} \in \PSD^d$ and a parameter $\kappa > 1$, we wish to determine
\begin{equation}\label{eq:mpcexact}\text{does there exist } w \in \R^n_{\ge 0} \text{ such that } \lmax\Par{\sum_{i \in [n]} w_i \ma_i} \le \kappa \lmin\Par{\sum_{i \in [n]} w_i \ma_i}?\end{equation}
We note that the problem \eqref{eq:mpcexact} is a special case of the more general mixed packing-covering semidefinite programming problem defined in \cite{JambulapatiLLPT20}, with packing matrices $\{\ma_i\}_{i \in [n]}$ and covering matrices $\{\kappa \ma_i\}_{i \in [n]}$; see discussion in Appendix~\ref{app:mpcdiscuss}. We define an $\eps$-\emph{tolerant tester} for the decision problem \eqref{eq:mpcexact} to be an algorithm which returns ``yes'' and a set of feasible weights whenever \eqref{eq:mpcexact} is feasible for the value $(1 - \eps)\kappa$, and ``no'' whenever it is infeasible for the value $(1 + \eps)\kappa$ (and can return either answer in the middle range). After developing such a tester, we apply it to our rescaling problems by incrementally searching for the optimal $\kappa$. It may be helpful to think of the parameter $\eps$ as a sufficiently small constant (e.g.\ $\frac 1 {10}$); it will be set this way in our applications.

To develop a tolerant tester for \eqref{eq:mpcexact}, we require access to an algorithm for solving the optimization variant of a pure packing SDP,
\begin{equation}\label{eq:optv}\opt(v) \defeq \max_{
		 w \in \R^n_{\ge 0}\,:\,\sum_{i \in [n]} w_i \ma_i \preceq \id
} v^\top w.\end{equation}
The pure packing SDP solver we use is based on combining a solver for the testing variant of \eqref{eq:optv} by \cite{Jambulapati0T20} with a binary search (we use the result in \cite{Jambulapati0T20} as it has the state-of-the-art dependence on $\eps$). We state its guarantees as Proposition~\ref{prop:optv}, and defer a proof to Appendix~\ref{app:slowoptimal}.

\begin{proposition}\label{prop:optv}
	Let $\opt_+$ and $\opt_-$ be known upper and lower bounds on $\opt(v)$ defined in \eqref{eq:optv}. There is an algorithm, $\algpack$, which succeeds with probability $\ge 1 - \delta$, whose runtime is 
	\[O\Par{\tmv\Par{\Brace{\ma_i}_{i \in [n]}} \cdot \frac{\log^2(ndT(\delta\eps)^{-1})\log^2 d}{\eps^5}} \cdot T \text{ for } T = O\Par{\log\log \frac{\opt_+}{\opt_-} + \log \frac 1 \eps},\]
	and returns an $\eps$-multiplicative approximation to $\opt(v)$, and $w$ attaining this approximation.
\end{proposition}

We require one additional tool, a regret analysis of matrix multiplicative weights from \cite{ZhuLO15}. This will be used in Algorithm~\ref{alg:decidempc} to certify that the constraints are met in the ``yes'' case.

\begin{proposition}[Theorem 3.1, \cite{ZhuLO15}]\label{prop:mmw}
Consider a sequence of gain matrices $\{\mg_t\}_{0 \le t < T} \subset \PSD^d$, which all satisfy for step size $\eta > 0$, $\norm{\eta \mg_t}_2 \le 1$. Then iteratively defining (from $\ms_0 \defeq \mzero$)
\[\my_t \defeq \frac{\exp(\ms_t)}{\Tr\exp(\ms_t)},\; \ms_{t + 1} \defeq \ms_t - \eta \mg_t,\]
we have the bound for any $\mmu \in \PSD^d$ with $\Tr(\mmu) = 1$,
\[\frac{1}{T}\sum_{0 \le t < T} \inprod{\mg_t}{\my_t - \mmu} \le \frac{\log d}{\eta T} + \frac 1 T \sum_{t \in [T]} \eta \norm{\mg_t}_2 \inprod{\mg_t}{\my_t}.\]
\end{proposition}

Finally, we are ready to state our $\eps$-tolerant tester for the decision problem \eqref{eq:mpcexact} as Algorithm~\ref{alg:decidempc}. Our algorithm is inspired by that in prior work \cite{Lee017, JambulapatiSS18}, which reduces mixed packing-covering instances to calls to a packing solver. More concretely, Algorithm~\ref{alg:decidempc} uses the guarantees of packing solvers in a matrix multiplicative weights framework to obtain covering guarantees.
\begin{algorithm}[ht!]
	\caption{$\DecideMPC(\{\ma_i\}_{i \in [n]}, \kappa, \algpack, \delta, \eps, \rho)$}
	\begin{algorithmic}[1]\label{alg:decidempc}
		\STATE \textbf{Input:} $\{\ma_i\}_{i \in [n]} \in \PSD^{d \times d}$ such that $\lmax(\ma_i) \in [1, \rho]$ for all $i \in [n]$, $\kappa > 1$, $\algpack$ which on input $v \in \R^n_{\ge 0}$ returns $w \in \R^n_{\ge 0}$ satisfying (recalling definition \eqref{eq:optv})
		\[\sum_{i \in [n]} w_i \ma_i \preceq \id,\; v^\top w \ge \Par{1 - \frac \eps {10}} \opt(v), \text{ with probability} \ge 1 - \frac{\delta}{2T} \text{ for some } T = O\Par{\frac{\kappa \log d}{\eps^2}},\]
		failure probability $\delta \in (0, 1)$, tolerance $\eps \in (0, 1)$
		\STATE \textbf{Output:} With probability $\ge 1 - \delta$: ``yes'' or ``no'' is returned. The algorithm must return ``yes'' if there exists $w \in \R^n_{\ge 0}$ with 
		\begin{equation}\label{eq:feasiblempc}\lmax\Par{\sum_{i \in [n]} w_i \ma_i} \le (1 - \eps)\kappa\lmin\Par{\sum_{i \in [n]} w_i \ma_i},\end{equation}
		and if ``yes'' is returned, a vector $w$ is outputted with
		\begin{equation}\label{eq:feasiblempcplus}\lmax\Par{\sum_{i \in [n]} w_i \ma_i} \le (1 + \eps) \kappa \lmin\Par{\sum_{i \in [n]} w_i \ma_i}.\end{equation}
		\STATE $\eta \gets \frac \eps {10\kappa}$, $T \gets \left\lceil \frac{10\log d}{\eta\eps} \right\rceil$, $\my_0 \gets \frac 1 d \id$, $\ms_0 \gets \mzero$
		\FOR{$0 \le t < T$}
		\STATE $\my_{t} \gets \frac{\exp(\ms_{t})}{\Tr\exp(\ms_{t})}$
		\STATE $v_t \gets$ entrywise nonnegative $(\frac \eps {10}, \frac{\eps}{10\kappa n})$-approximations to $\{\inprod{\ma_i}{\my_t}\}_{i \in [n]}$, with probability $\ge 1 - \frac{\delta}{4T}$
		\STATE $x_t \gets \algpack(\kappa v_t)$
		\STATE $\mg_{t} \gets \kappa \sum_{i \in [n]} [x_t]_i \ma_i$
		\IF{$\kappa \inprod{x_t}{v_t} < 1 - \frac \eps 5$}
		\RETURN ``no''
		\ENDIF
		\STATE $\ms_{t + 1} \gets \ms_t - \eta \mg_t$
		\STATE $\tau \gets \frac{\log d}{\eps}$-additive approximation to $\lam_{\min}(-\ms_{t + 1})$, with probability $\ge 1 - \frac{\delta}{4T}$
		\IF{$\tau \ge \frac{12\log d}{\eps}$}
		\RETURN (``yes'', $\bx$) for $\bx \defeq \frac 1 {t + 1} \sum_{0 \le s \le t} x_s$
		\ENDIF
		\ENDFOR
		\RETURN (``yes'', $\bx$) for $\bx \defeq \frac 1 T \sum_{0 \le t < T} x_t$
	\end{algorithmic}
\end{algorithm}

We begin by giving a correctness proof of Algorithm~\ref{alg:decidempc}, and follow with a runtime analysis under using the subroutine $\algpack$ of Proposition~\ref{prop:optv} for Line 7. We discuss further computational issues regarding implementing Lines 6 and 13 as they arise in specific applications.

\begin{lemma}\label{lem:decidempccorrect}
Algorithm~\ref{alg:decidempc} meets its output guarantees (as specified on Line 2).
\end{lemma}
\begin{proof}
Throughout, assume all calls to $\algpack$ and the computation of approximations as given by Lines 6 and 13 succeed. By union bounding over $T$ iterations, this gives the failure probability.

We next show that if the algorithm terminates on Line 15, it is always correct. By the definition of $\algpack$, all $\mg_t \preceq \kappa \id$, so throughout, $-\ms_{t + 1} \preceq \eta \kappa T \id \preceq \frac{11\kappa\log d}{\eps}$. If the check on Line 14 passes, we must have $-\ms_{t + 1} \succeq \frac{11\log d}{\eps} \id$, and hence the matrix $-\frac{1}{t + 1}\ms_{t + 1}$ has condition number at most $\kappa$. The conclusion follows as $\sum_{i \in [n]} \bx_i \ma_i = -\frac{1}{(t + 1)\eta\kappa}\ms_{t + 1}$ has the same condition number as $-\frac{1}{t+1}\ms_{t + 1}$.

We next prove correctness in the ``no'' case. Suppose the problem \eqref{eq:feasiblempc} is feasible; we show that the check in Line 9 will never pass (so the algorithm never returns ``no''). Let $v^\star_t$ be the vector which is entrywise exactly $\{\inprod{\ma_i}{\my_t}\}_{i \in [n]}$, and let $v'_t$ be a $\frac{\eps}{10}$-multiplicative approximation to $v^\star_t$ such that $v_t$ is an entrywise $\frac{\eps}{10n}$-additive approximation to $v'_t$. By befinition, it is clear $\opt(\kappa v'_t) \ge (1 - \frac \eps {10})\opt(\kappa v^\star_t)$. Moreover, by the assumption that all $\lmax(\ma_i) \ge 1$, all $w_i \le 1$ in the feasible region of the problem \eqref{eq:optv}. Hence, the combined additive error incurred by the approximation $\inprod{\kappa v_t}{w}$ to $\inprod{\kappa v'_t}{w}$ for any feasible $w$ is $\frac{\eps}{10}$. Altogether, by the guarantee of $\algpack$,
\begin{equation}\label{eq:optrelate}\kappa \inprod{v_t}{x_t} \ge \Par{1 - \frac \eps {10}}^2 \opt(\kappa v^\star_t) - \frac{\eps}{10}, \text{ where } \opt(\kappa v^\star_t) = \max_{\substack{\sum_{i \in [n]} w_i \ma_i \preceq \id \\ w \in \R^n_{\ge 0}}} \kappa \inprod{\my_t}{\sum_{i \in [n]} w_i \ma_i}.\end{equation}
However, by feasibility of \eqref{eq:feasiblempc} and scale invariance, there exists a $w \in \R^n_{\ge 0}$ with $\sum_{i \in [n]} w_i \ma_i \preceq \id$ and $(1 - \eps) \kappa \sum_{i \in [n]} w_i \ma_i \succeq \id$. Since $\my_t$ has trace $1$, this certifies $\opt(\kappa v_t^\star) \ge \frac 1 {1 - \eps}$, and thus \[\kappa\inprod{v_t}{x_t} \ge \Par{1 - \frac \eps {10}}^2 \cdot \frac 1 {1 - \eps} - \frac{\eps}{10} > 1 - \frac \eps 5.\]
Hence, whenever the algorithm returns ``no'' it is correct. Assume for the remainder of the proof that ``yes'' is returned on Line 18. Next, we observe that whenever $\alg$ succeeds on iteration $t$, $\sum_{i \in [n]} [x_t]_i \ma_i \preceq \id$, and hence in every iteration we have $\norm{\mg_t}_2 \le \kappa$. Proposition~\ref{prop:mmw} then gives
\[\frac{1}{T} \sum_{0 \le t < T} \inprod{\mg_t}{\my_t - \mmu} \le \frac{\log d}{\eta T} + \frac 1 T \sum_{t \in [T]} \eta \norm{\mg_t}_2 \inprod{\mg_t}{\my_t},\text{ for all } \mmu \in \PSD^d \text{ with } \Tr(\mmu) = 1.\]
Rearranging the above display, using $\eta \norm{\mg_t}_2 \le \frac \eps {10}$, and minimizing over $\mmu$ yields
\[\lmin\Par{\frac 1 T \sum_{0 \le t < T} \mg_t} \ge \frac {1 - \frac \eps {10}} T \sum_{0 \le t < T} \inprod{\mg_t}{\my_t} - \frac{\log d}{\eta T} \ge \frac {1 - \frac \eps {10}} T \sum_{0 \le t < T} \inprod{\mg_t}{\my_t} - \frac \eps {10}.\]
The last inequality used the definition of $T$. However, by definition of $v_t$, we have for all $0 \le t < T$,
\begin{equation}\label{eq:ygbound}\inprod{\my_t}{\mg_t} = \kappa \sum_{i \in [n]} [x_t]_i \inprod{\ma_i}{\my_t} \ge \Par{1 - \frac \eps {10}} \kappa \inprod{x_t}{v_t} \ge \Par{1 - \frac \eps {10}}\Par{1 - \frac \eps 5} \ge 1 - \frac {3\eps}{10}.\end{equation}
The second-to-last inequality used that Line 9 did not pass. Combining the previous two displays,
\[\kappa\lmin\Par{\sum_{i \in [n]} \bx_i \ma_i} = \lmin\Par{\frac 1 T \sum_{0 \le t < T} \mg_t} \ge \Par{1 - \frac \eps {10}}\Par{1 - \frac {3\eps}{10}} - \frac \eps {10} \ge 1 - \frac \eps 2.\]
On the other hand, since all $0 \le t < T$ have $\sum_{i \in [n]} [x_t]_i \ma_i \preceq \id$, by convexity $\sum_{i \in [n]} \bx_i \ma_i \preceq \id$. Combining these two guarantees and $(1 + \eps)(1 - \frac \eps 2) \ge 1$ shows $\bx$ is correct for the ``yes'' case.
\end{proof}

We remark that the proof of the ``no'' case in Lemma~\ref{lem:decidempccorrect} demonstrates that in all calls to $\alg$, we can set our lower bound $\opt_- = 1 - O(\eps)$, since the binary search of Proposition~\ref{prop:optv} will never need to check smaller values to determine whether the test on Line 9 passes. On the other hand, the definition of $\opt(\kappa v^\star_t)$ in \eqref{eq:optrelate}, as well as $\opt(\kappa v_t) \le (1 + \frac \eps {10})\opt(\kappa v^\star_t)$ by the multiplicative approximation guarantee, shows that it suffices to set $\opt_+ \le (1 + O(\eps)) \kappa$. Using these bounds in the context of Proposition~\ref{prop:optv} implies the following bounds on the cost of $\algpack$ in Algorithm~\ref{alg:decidempc}.

\begin{corollary}\label{cor:algcostmpc}
Using the algorithm of Proposition~\ref{prop:optv} as $\algpack$ in Algorithm~\ref{alg:decidempc} with $\frac{\opt_+}{\opt_-} = O(\kappa)$, for any sufficiently small constant $\eps > 0$, the cost of each call to $\algpack$ is
\[O\Par{\tmv\Par{\Brace{\ma_i}_{i \in [n]}} \cdot \log^2\Par{\frac{nd\kappa}{\delta}} \log^2(d) \log\log\kappa}.\] 
\end{corollary}

\subsection{Constant-factor optimal inner scaling}\label{ssec:kslslow}

In this section, we show how to use Algorithm~\ref{alg:decidempc} to efficiently compute a reweighting $w$ such that for $\mw \defeq \diag{w}$, $\kappa(\ma^\top \mw \ma) = O(\ksl(\ma))$ for full-rank $\ma \in \R^{n \times d}$ with $n \ge d$. In particular, this implies $w$ is a constant-factor optimal inner scaling. We use $a_i$ to denote row $\ai$ in this section, and assume by scale invariance in the definition of reweightings $w$ that all $\norm{a_i}_2 = 1$. We also define the following rank-one matrices in $\PSD^d$:
\begin{equation}\label{eq:leftmats}\ma_i \defeq a_ia_i^\top, \text{ for all } i \in [n].\end{equation}
By observation, all $\ma_i$ satisfy the eigenvalue requirement of Algorithm~\ref{alg:decidempc} via $\norm{a_i}_2 = 1$. The following two sections respectively demonstrate how to efficiently implement Lines 13 and 6 of Algorithm~\ref{alg:decidempc}, by using polynomial approximations to the exponential and random projections.

\subsubsection{Estimating the smallest eigenvalue}

We discuss the computation of the approximate smallest eigenvalue of a matrix $\mm$. At a high level, our strategy is to use power method on the negative exponential $\exp(-\mm)$, which we approximate to additive error via Fact~\ref{fact:ratexp}. We first state a guarantee on the power method from \cite{MuscoM15}.

\begin{fact}[Theorem 1, \cite{MuscoM15}]\label{fact:powermethod}
	For any $\delta \in (0, 1)$ and $\mm \in \Sym_{\ge 0}^d$, there is an algorithm, $\Power(\mm, \delta)$, which returns with probability at least $1 - \delta$ a value $V$ such that $\lmax(\mm) \ge V \ge 0.9 \lmax(\mm)$. The algorithm runs in time $O(\tmv(\mm)\log \frac d \delta)$, and is performed as follows:
	\begin{enumerate}
		\item Let $u \in \R^d$ be a random unit vector.
		\item For some $\Delta = O(\log \frac d \delta)$, let $v \gets \frac{\mm^\Delta u}{\norm{\mm^\Delta u}_2}$.
		\item Return $\norm{\mm v}_2$.
	\end{enumerate}
\end{fact}

We next state our main technical tool, a low-degree approximation to the inverse exponential of a bounded positive semidefinite matrix with additive error, following immediately from Fact~\ref{fact:polyexp}. 

\begin{corollary}\label{cor:approxexp}
Given $R > 1$, $\mm \in \PSD^d$, and $\kappa$ with $\mm \preceq \kappa \id$, we can compute a degree-$O(\sqrt{\kappa R} + R)$ polynomial $p$ such that for $\mpack = p(\mm)$,
\[\exp\Par{-\mm} - \exp\Par{-R}\id \preceq \mpack \preceq \exp\Par{-\mm} + \exp\Par{-R}\id.\]
\end{corollary}

Combining the previous two results, we obtain our smallest eigenvalue approximation.

\begin{lemma}\label{lem:approxlmin}
Given $R > 1$, $\delta \in (0, 1)$, $\mm \in \PSD^d$, and $\kappa$ with $\mm \preceq \kappa\id$, we can compute a $R$-additive approximation to $\lmin(\mm)$ with probability $\ge 1 - \delta$ in time
\[O\Par{\tmv(\mm) \cdot \Par{\sqrt{\kappa R} +R} \log\Par{\frac{d}{\delta}}}.\]
\end{lemma}
\begin{proof}
Our algorithm is to apply the power method to the polynomial approximation in Corollary~\ref{cor:approxexp}. By shifting the definition of $R$ in Corollary~\ref{cor:approxexp} by a constant, correctness follows from the guarantees of Fact~\ref{fact:powermethod}. The runtime combines the degree in Corollary~\ref{cor:approxexp} with the overhead of Fact~\ref{fact:powermethod}.
\end{proof}

\subsubsection{Estimating inner products with a negative matrix exponential}

We next discuss computational issues regarding Line 6 of Algorithm~\ref{alg:decidempc}. We begin by handling approximation of $\Tr\exp(-\mm)$ for $\mm \in \PSD^d$ with bounded smallest and largest eigenvalues.

\begin{lemma}\label{lem:trexpnegm}
Given $\mm \in \PSD^d$, $R, \kappa > 0$ such that $\lmin(\mm) \le R$ and $\lmax(\mm) \le \kappa R$, $\delta \in (0, 1)$, and sufficiently small constant $\eps > 0$, we can compute an $\eps$-multiplicative approximation to $\Tr\exp(-\mm)$ with probability $\ge 1 - \delta$ in time
\[O\Par{\tmv(\mm) \cdot \Par{\sqrt{\kappa R} + R} \log\Par{\frac d \delta}}.\]
\end{lemma}
\begin{proof}
First, with probability at least $1 - \delta$, choosing $k =O(\log \frac d \delta) $ in Fact~\ref{fact:jl} and taking a union bound guarantees that for all rows $j \in [d]$, we have 
\[\norm{\mq \Brack{\exp\Par{-\half \mm}}_{j:}}_2^2 \text{ is a } \frac \eps 3\text{-multiplicative approximation of } \norm{\Brack{\exp\Par{-\half \mm}}_{j:}}_2^2.\]
Condition on this event in the remainder of the proof. Since
\[\Tr\exp(-\mm) = \sum_{j \in [d]} \norm{\Brack{\exp\Par{-\half \mm}}_{j:}}_2^2,\]
and the following equalities hold
\begin{align*}
\sum_{j \in [d]} \norm{\mq \Brack{\exp\Par{-\half \mm}}_{j:}}_2^2 &= \Tr\Par{\exp\Par{-\half \mm} \mq^\top \mq \exp\Par{-\half \mm}} \\
&= \Tr\Par{\mq \exp\Par{-\mm} \mq^\top} = \sum_{\ell \in [k]} \norm{\exp\Par{-\half \mm} \mq_{\ell:}}_2^2,
\end{align*}
it follows that it suffices to obtain a $\frac \eps 3$-multiplicative approximation to the last sum in the above display. Since $\Tr\exp(-\mm) \ge \exp(-R)$ by the assumption on $\lmin(\mm)$, it then suffices to approximate each term $\norm{\exp(-\half \mm) \mq_{\ell:}}_2^2$ to an additive $\frac{\eps}{3k}\exp(-R)$. For simplicity, fix some $\ell \in [k]$ and denote $q \defeq \mq_{\ell:}$; recall $\norm{q}_2^2 = \frac 1 k$ from the definition of $\mq$ in Fact~\ref{fact:jl}.

By rescaling, it suffices to demonstrate that on any unit vector $q \in \R^d$, we can approximate $\norm{\exp(-\half \mm)q}_2^2$ to an additive $\frac \eps 3 \exp(-R)$. To this end, we note that (after shifting the definition of $R$ by a constant) Corollary~\ref{cor:approxexp} provides a matrix $\mpack$ with $\tmv(\mpack) = O(\tmv(\mm) \cdot (\sqrt{\kappa R}+ R))$ and
\[\exp\Par{-\mm} - \frac \eps 3 \exp\Par{-R}\id \preceq \mpack \preceq \exp\Par{-\mm} + \frac \eps 3\exp\Par{-R}\id,\]
which exactly meets our requirements by taking quadratic forms. By setting $\kappa \gets \kappa R$ and adjusting the definition of $R$ by a constant in Corollary~\ref{cor:approxexp}, the runtime follows from the cost of applying $\mpack$ to each of the $k = O(\log \frac d \delta)$ rows of $\mq$.
\end{proof}

We next handle the required approximation of inner products $\inprod{a_ia_i^\top}{\exp(-\mm)}$.

\begin{lemma}\label{lem:line5left}
Given $\mm \in \Sym_{\ge 0}^d$ and $\kappa$ with $\mm \preceq \kappa\id$, $c > 1$, $\delta \in (0, 1)$, and sufficiently small constant $\eps$, we can compute $(\eps, \exp(-c))$-approximations to all
\[\Brace{\inprod{a_ia_i^\top}{\exp(-\mm)}}_{i \in [n]} \text{ where } \Brace{a_i}_{i \in [n]} \text{ are rows of } \ma \in \R^{n \times d} \text{ and } \norm{a_i}_2 = 1 \text{ for all } i \in [d],\]
with probability $\ge 1 - \delta$ in time
\[O\Par{\Par{\tmv(\mm) \cdot \Par{\sqrt{\kappa c} + c} + \nnz(\ma)} \log \frac n \delta}.\]
\end{lemma}
\begin{proof}
As in the proof of Lemma~\ref{lem:trexpnegm}, by taking a union bound it suffices to sample a $\mq \in \R^{k \times d}$ for $k = O(\log \frac n \delta)$ and instead compute all $\norm{\mq \exp(-\half \mm) a_i}_2^2$ to additive error $\exp(-c)$. By appropriately renormalizing by $k$, letting $q$ be some row of $\mq$, it suffices to show instead how to compute for an arbitrary unit vector $q \in \R^d$, the quantity $\inprod{q}{\exp(-\half \mm)a_i}^2$ to additive error $\exp(-c)$.

To this end, consider the use of the approximation $\inprod{q}{\mpack a_i}^2$ for some $\mpack$ with $-\frac 1 3 \exp(-c) \id \preceq \mpack - \exp(-\half \mm) \preceq \frac 1 3 \exp(-c) \id$. Letting the difference matrix be $\md \defeq \mpack - \exp(-\half \mm)$, we compute
\begin{align*}
\Par{q^\top \exp\Par{-\half \mm} a_i}^2 - \Par{q^\top \mpack a_i}^2 &= 2\Par{q^\top \exp\Par{-\half \mm} a_i}\Par{q^\top \md a_i} + \Par{q^\top \md a_i}^2\\
&\le 2\norm{\md}_2 + \norm{\md}_2^2 \le \exp(-c).
\end{align*}
We used $q$ and $a_i$ are unit vectors, $\exp(-\half \mm) \preceq \id$, and $\norm{\md}_2 \le \frac 1 3 \exp(-c) \le 1$. Hence, $\inprod{q}{\mpack a_i}^2$ is a valid approximation. The requisite $\mpack$ is given by Corollary~\ref{cor:approxexp} with $\tmv(\mpack) = O(\tmv(\mm) \cdot (\sqrt{\kappa c} + c))$.
The runtime follows from applying $\mpack$ to rows of $\mq$ to form $\widetilde{\mq}$, and then computing all $\|\widetilde{\mq} a_i\|_2^2$.
\end{proof}

\subsubsection{Putting it all together}

By combining Lemma~\ref{lem:line5left} with Corollary~\ref{cor:algcostmpc}, we have the following guarantee on the cost of running Algorithm~\ref{alg:decidempc} for a given value of $\kappa$ with the matrices defined in \eqref{eq:leftmats}.

\begin{lemma}\label{lem:alg1left}
The cost of running Algorithm~\ref{alg:decidempc} for some $\kappa > 0$, sufficiently small constant $\eps$, and the matrices defined in \eqref{eq:leftmats} is
\[O\Par{\nnz(\ma)\cdot \kappa^{1.5} \log^{2}(d)\log^2\Par{\frac{n\kappa}{\delta}}\log\log\kappa }.\]
\end{lemma}
\begin{proof}
The cost of all lines other than Line 6 and 13 can clearly be seen to fit in the runtime budget, by Corollary~\ref{cor:algcostmpc} (note that we do not explicitly compute the matrices $\mg_t$, $\ms_t$, or $\my_t$ in any iteration). 

To bound the cost of Line 13, it suffices to use Lemma~\ref{lem:approxlmin} with $\kappa \gets O(\frac{\kappa \log d}{\eps})$ and $R \gets \frac{\log d}{\eps}$. To see this is valid, this $R$ value is the additive approximation required by Line 13, and the largest eigenvalue of $-\ms_{t + 1}$ is bounded by $\eta \kappa T = O(\frac{\kappa \log d}{\eps})$ by the guarantees of $\algpack$.

To bound the cost of Line 6, we use Lemmas~\ref{lem:trexpnegm} and~\ref{lem:line5left}. First, note that at all times Line 6 is run, it suffices to parameterize $R \gets \frac{13\log d}{\eps}$ and $\kappa \gets O(\kappa)$ in Lemma~\ref{lem:trexpnegm}; the former guarantee follows from the algorithm not terminating on Line 15, and the latter follows from our earlier eigenvalue bound. We then compute a $\frac \eps {30}$-multiplicative approximation to the denominator $\Tr\exp(\ms_{t + 1})$ of all inner products required in Line 6 within the designated time budget.

Finally, we claim it suffices to parameterize Lemma~\ref{lem:line5left} with $\eps \gets \frac{\eps}{30}$ and $c \gets R + \log\frac{10\kappa n}{\eps}$, where $R \gets \frac{13\log d}{\eps}$ as before. To see this, the composition of two $\frac \eps {30}$-multiplicative approximations (one due to the error of estimating the denominator) is a $\frac \eps {10}$-multiplicative approximation, and since the denominator is at least $\exp(-R)$, the additive error can only be blown up by a factor of $\exp(R)$. For the given setting of $c$, this amounts to an overall $\frac{\eps}{10\kappa n}$-additive approximation (after normalization), as desired. The runtime of Lemma~\ref{lem:line5left} can again be seen to fit in the designated budget.
\end{proof}

Finally, by implementing an incremental search using Lemma~\ref{lem:alg1left}, we have our main result.

\begin{theorem}\label{thm:leftslow}
There is an algorithm, which given full-rank $\ma \in \R^{n \times d}$ for $n \ge d$ computes $w \in \R^n_{\ge 0}$ such that $\kappa(\ma^\top \mw \ma) = (1 + \eps)\ksl(\ma)$ for arbitrarily small $\eps = \Theta(1)$, with probability $\ge 1 - \delta$ in time
\[O\Par{\nnz(\ma)\cdot \Par{\ksl(\ma)}^{1.5} \cdot \log^{2}(d)\log^2\Par{\frac{n\ksl(\ma)}{\delta}}\log\log\Par{\ksl(\ma)}}.\]
\end{theorem}
\begin{proof}
Throughout let $\ksl \defeq \ksl(\ma)$ for notational simplicity. For a given value of $\kappa \ge \frac 1 {1 - \eps} \ksl$, it is clear from the definition of $\ksl$ that \eqref{eq:feasiblempc} is feasible and the algorithm will return ``yes'' with a weighting yielding a condition number at most a $\frac{1 + \eps}{1 - \eps}$ factor worse than $\ksl$. By starting with the guess $\kappa = 1$ and incrementing by factors of $1 + \eps$ each time, the algorithm will clearly terminate on $\kappa = (1 + O(\eps))\ksl$ in $O(\log \ksl)$ calls. Note that this strategy does not add a $\log \ksl$ overhead to the runtime, because it exhibits geometric decay and hence only has a constant overhead. Since $O(\ksl)$ is an upper bound on all $\kappa$ that Algorithm~\ref{alg:decidempc} uses, the claim follows from Lemma~\ref{lem:alg1left}.
\end{proof}

\subsection{Constant-factor optimal outer scaling with a factorization}\label{ssec:ksrslow}

In this section, we show how to modify the strategy of Section~\ref{ssec:kslslow} to handle outer scalings. Throughout, let $\ma \in \R^{n \times d}$ with $n \ge d$ be full-rank, and define a matrix $\mb \in \R^{n \times n}$ as follows: let its first $d$ rows be the columns of $\ma$, and let its last $n - d$ rows be any basis for the orthogonal complement of the column span of $\ma$, such that $\mb$ is non-singular. We assume without loss of generality (by scale invariance) that the columns of $\ma$ have been scaled to have unit norm.

\begin{lemma}\label{lem:bsuffices}
We have $\ksl(\mb) = \ksr(\ma^\top\ma)$. Moreover, for any $w \in \R^n_{\ge 0}$, letting $\tw \in \R^d_{\ge 0}$ keep its first $d$ entries, we have that $\kappa(\tmw^\half \ma^\top \ma \tmw^\half) \le \kappa(\mb^\top \mw \mb)$.
\end{lemma}
\begin{proof}
We first prove the second statement. Since the eigenvalues of $\ma \tmw \ma^\top$ are the same as those of $\tmw^\half \ma^\top \ma \tmw^\half$ with the addition of $0$ as an eigenvalue with multiplicity $n - d$, and the nonzero eigenvalues of $\ma \tmw \ma^\top$ are a subset of the eigenvalues of $\mb^\top \mw \mb$ since the rows of $\mb$ separate into two orthogonal subspaces of $\R^n$, this yields the claim.

The second statement clearly implies $\ksl(\mb) \ge \ksr(\ma^\top\ma)$, so it remains to show $\ksl(\mb) \le \ksr(\ma^\top\ma)$. Given $\tw \in \R^d_{\ge 0}$, it suffices to show how to extend $\tw$ to a vector $w \in \R^n_{\ge 0}$ such that $\kappa(\mb^\top \mw \mb) \le \kappa(\tmw^\half \ma^\top \ma \tmw^\half)$. To this end, let $\lam \defeq \lmin(\tmw^\half \ma^\top \ma \tmw^\half)$, and consider the $w$ which has last $n - d$ coordinates equal to $\lam$, and first $d$ coordinates equal to $\tw$. The eigenvalues of $\mb^\top \mw \mb$ are then all either $\lam$, or eigenvalues of $\tmw^\half \ma^\top \ma \tmw^\half$, yielding the claim.
\end{proof}

Lemma~\ref{lem:bsuffices} implies that to solve the optimal outer scaling problem for $\ma$, it suffices to solve the optimal inner scaling problem for $\mb$. However, in general it may be too expensive to write down the matrix $\mb$ and perform computations with it explicitly. In the remainder of this section, we show how to implement the steps of Algorithm~\ref{alg:decidempc} using only implicit access to $\mb$.

We define the following notation: for $j \in [d]$ let $a_j \defeq \aj$ (column $j$ of $\ma$), and for $i \in [m]$ for $m \defeq n -d$, let $u_i$ be the $(d + i)^{\text{th}}$ row of $\mb$, such that vertically concatenating $\{a_j\}_{j \in [d]}$ and $\{u_i\}_{i \in [m]}$ forms $\mb$. Finally, for $i \in [n]$ in this section we let
\begin{equation}\label{eq:rightmats}\ma_i \defeq \mb_{i:}\mb_{i:}^\top.\end{equation}
We begin by discussing the implicit implementation of Line 7 of Algorithm~\ref{alg:decidempc}.

\begin{lemma}\label{lem:line6right}
Given a vector $v$, to compute a $\alg(v)$ satisfying the requirements of Line 7 of Algorithm~\ref{alg:decidempc} using the matrices \eqref{eq:rightmats}, letting $\tv$ be the first $d$ coordinates of $v$, it suffices to evaluate $\tx \gets \alg(\tv)$ and return $x$ whose first $d$ coordinates are $\tx$, and whose last $n - d$ coordinates are $1$.
\end{lemma}
\begin{proof}
Recall that $\alg$ guarantees a $\frac \eps {10}$-multiplicative approximation to
\[\opt(v) \defeq \max_{\substack{\sum_{i \in [n]} w_i \ma_i \preceq \id \\ w \in \R^n_{\ge 0}}} v^\top w,\]
as well as a satisfying $w$. Because the spectral constraint $\sum_{i \in [n]} w_i \ma_i \preceq \id$ for the matrices \eqref{eq:rightmats} decomposes into a constraint on the span of $\ma$ and a constraint on its complement, it suffices to provide $\frac \eps {10}$-multiplicative approximations to
\[\max_{\substack{\sum_{i \in [d]} w_i \ma_i \preceq \id \\ w \in \R^d_{\ge 0}}} \tv^\top w \text{ and } \max_{\substack{\sum_{i \in [n] \setminus [d]} w_i \ma_i \preceq \id \\ w \in \R^{n - d}_{\ge 0}}} (v')^\top w,\]
where $v'$ is the last $n - d$ coordinates of $v$. By definition of $\tx$, it is a $\frac \eps {10}$-multiplicative approximation to the former problem. By nonnegativity of $v'$ and since the last $n - d$ rows of $\mb$ are an orthonormal basis, the all-ones vector solves the second problem optimally. Concatenating yields the result.
\end{proof}

Given the result of Lemma~\ref{lem:line6right}, if we choose to implement Line 7 of Algorithm~\ref{alg:decidempc} as described, in every iteration $t$ the last $n - d$ coordinates of $x_t$ are identical. By combining this invariant with Lemma~\ref{lem:line5left}, we give an implicit implementation of Lines 6 and 13 of Algorithm~\ref{alg:decidempc}.

\begin{lemma}\label{lem:line5right}
Suppose in some iteration, every $x_s$ for $0 \le s < t$ has its last $n - d$ coordinates identical and explicitly known. Then we can implement Lines 6 and 13 of Algorithm~\ref{alg:decidempc} in the same runtimes (up to constants) as given in Lemmas~\ref{lem:approxlmin},~\ref{lem:trexpnegm}, and~\ref{lem:line5left}.
\end{lemma}
\begin{proof}
Consider the matrix $\ms_t \defeq -\sum_{0 \le s < t} \eta \mg_s$ which defines each $\my_t$ by normalized negative exponentiation. Further, define the matrices in each iteration $0 \le s < t$,
\[\mg_s^{(1)} \defeq \kappa \sum_{i \in [d]} [x_s]_i \ma_i,\; \mg_s^{(2)} \defeq \kappa \sum_{i \in [n] \setminus [d]} [x_s]_i \ma_i = C \id_{\perp},\]
such that $\mg_s = \mg_s^{(1)}  + \mg_s^{(2)}$. Here, we let $\id_{\perp}$ be the projection matrix onto the orthogonal complement of the column span of $\ma$, and we used our invariant on the last coordinates of each $x_s$ to explicitly compute $C$. We can then analogously define
\[\ms_t^{(1)} \defeq -\sum_{0 \le s < t} \eta \mg_s^{(1)},\; \ms_t^{(2)} \defeq -\sum_{0 \le s < t} \eta \mg_s^{(2)} = C' \id_{\perp},\]
for some explicitly computable $C'$. Now, since $\exp$ preserves eigenspaces, for all $j \in [d]$,
\[a_j^\top \exp(\ms_t) a_j = a_j^\top \exp\Par{\ms_t^{(1)}} a_j,\]
which we can use Lemma~\ref{lem:line5left} to approximate since we have explicit access to all $\ma_j$ for $j \in [d]$. Moreover, for all $i \in [m]$,
\[u_i^\top \exp\Par{\ms_t} u_i = u_i^\top \exp\Par{C' \id_{\perp}} u_i = \exp(C'),\]
by the same logic. Next, to compute $\Tr\exp(\ms_t)$, by separability of eigenspaces it is
\[\Tr\exp\Par{\ms_t^{(1)}} + \Tr\exp\Par{\ms_t^{(2)}} - n.\]
Here we accounted for double-counting of zero eigenvalues. Lemma~\ref{lem:line5left} approximates the first summand, and we can explicitly compute the second summand as $(n - d) \exp(C')$. None of these computations is the runtime bottleneck, yielding the claim. This concludes implementing Line 6.

Finally, to implement Line 13 it suffices to run Lemma~\ref{lem:approxlmin} on just the matrix $\ms_t^{(1)}$, which we have explicit access to (making sure to project into the column span of $\ma$ when implementing the power method), and then check the value of $-C'$, since these are all the eigenvalues of $-\ms_t$.
\end{proof}

We combine these pieces to give our main outer scaling result under a factorization.

\begin{theorem}\label{thm:rightslow}
There is an algorithm, which given full-rank $\ma \in \R^{n \times d}$ for $n \ge d$ computes $w \in \R^n_{\ge 0}$ such that $\kappa(\mw^\half \ma^\top \ma \mw^\half) = (1 + \eps) \ksr(\ma^\top \ma)$ for arbitrarily small $\eps = \Theta(1)$, with probability $\ge 1 - \delta$ in time (where $\mk \defeq \ma^\top\ma$)
\[O\Par{\nnz(\ma)\cdot \Par{\ksr(\mk)}^{1.5} \cdot \log^{2}(d)\log^2\Par{\frac{n\ksr(\mk)}{\delta}}\log\log\Par{\ksr(\mk)}}.\]
\end{theorem}
\begin{proof}
By Lemma~\ref{lem:bsuffices}, it suffices to compute a near-optimal inner scaling of $\mb$ and then truncate to the first $d$ coordinates, which Theorem~\ref{thm:leftslow} accomplishes. To give the runtime, applying Lemma~\ref{lem:line6right} inductively satisfies the assumption of Lemma~\ref{lem:line5right}. The runtime then follows by combining Lemmas~\ref{lem:line6right} and~\ref{lem:line5right}, which reduces the computation to be bottlenecked by the runtime stated in Theorem~\ref{thm:leftslow}.
\end{proof} 	%

\section{Constant-factor optimal outer scalings without a factorization}
\label{sec:kernel}

In this section, we provide a method for computing an outer scaling which approximately obtains the optimal preconditioning quality $\ksk(\mk)$, for a matrix $\mk \in \PD^d$. To outline our approach, let $\ma \in \PD^d$ denote the symmetric square root of $\mk$ (which we recall we cannot explicitly access). At a high level, the difficulty is in efficiently simulating access to the rows of $\ma$, because if we had $\ma$ explicitly we could apply the algorithm of Section~\ref{ssec:ksrslow} (indeed, the additional complication of ``completing the basis'' is not even needed, as $\ma$ is square). However, to simulate access to $\ma$ requires efficient approximations to the square root, which in turn depend on the conditioning of $\mk$ (cf.\ Fact~\ref{fact:poly_approx_sqrt}) --- this is exactly the quality of $\mk$ which we aimed to improve!

We give a recursive approach to breaking this chicken-and-egg problem, inspired by similar approaches used in the literature, e.g.\ \cite{LiMP13, KapralovLMMS14, CohenLMMPS15}. In particular, the final algorithm we give in Section~\ref{ssec:homotopy} for approximating $\ksk(\mk)$ proceeds in logarithmically many \emph{phases}, each of which computes a reweighting approximating $\ksk(\mk + \lam \id)$ for a different value of $\lam$. By repeatedly halving $\lam$, we can use the reweighting for $\mk + \lam \id$ as a reweighting for the next phase $\mk + \frac \lam 2 \id$, incurring only a slight loss in condition number. We can then approximate the square root of the (reweighted) $\mk + \frac \lam 2 \id$, allowing us to efficiently implement Algorithm~\ref{alg:decidempc} for the next phase. The key technical observation underpinning this ``homotopy method'' is the following, which states that adding a multiple of the identity cannot increase the value of $\ksk$.

\begin{lemma}\label{lem:idcanthurt}
		For any matrix $\mk \in \PD^d$ and $\lam \ge 0$, $\ksk(\mk + \lam \id) \le \ksk(\mk)$.
	\end{lemma}
	\begin{proof}
		By scaling $\mk$ by $\lam$ appropriately (since $\ksk$ is invariant under scalar multiplication), it suffices to take $\lam = 1$. The definition of $\ksk$ implies there exists a diagonal matrix $\mw$ such that
		\begin{equation}\label{eq:equivdiag}\id \preceq \mw^\half \mk \mw^\half \preceq \ksk(\mk) \id \iff \mw^{-1} \preceq \mk \preceq \ksk(\mk) \mw^{-1}.\end{equation}
		Thus, to demonstrate $\ksk(\mk + \id) \le \ksk(\mk)$ it suffices to exhibit a diagonal $\tmw$ such that
		\[\tmw \preceq \mk + \id \preceq \ksk(\mk) \tmw. \]
		We choose $\tmw = \mw^{-1} + \id$; then, the above display follows from \eqref{eq:equivdiag} and $\id \preceq \id \preceq \ksk(\mk) \id$.
	\end{proof}

In Section~\ref{ssec:implsqrt}, we demonstrate how to implement the steps of Algorithm~\ref{alg:decidempc} with $\{\ma_i = a_ia_i^\top\}_{i \in [d]}$, where $\{a_i\}_{i \in [d]}$ are rows of an unknown matrix $\ma$, and $\mk = \ma^2$ has bounded condition number and is explicitly given. We then use this primitive to implement our homotopy method in Section~\ref{ssec:homotopy}.

\subsection{Matrix exponential operations with implicit square root access}\label{ssec:implsqrt}

In this section, we demonstrate that we can perform the operations required by Lines 6, 7, and 13 of Algorithm~\ref{alg:decidempc} using implicit square root access with multiplicative accuracy. Specifically, we will repeated use the following corollary of Fact~\ref{fact:poly_approx_sqrt} to simulate square root access.

\begin{corollary}\label{cor:implsqrt}
For any vector $b \in \R^d$, $\delta, \eps \in (0, 1)$, and $\mm \in \PD^d$ with $\kappa(\mm) \le \kscale$, with probability $\ge 1 - \delta$ we can compute $u \in \R^d$ such that
\[\norm{u - \mm^\half b}_2 \le \eps \norm{\mm^\half b}_2 \text{ in time } O\Par{\tmv\Par{\mm} \cdot \Par{\sqrt{\kscale} \log \frac{\kscale}{\eps} + \log \frac d \delta}}.\]
\end{corollary}
\begin{proof}
First, we compute a $2$-approximation to $\mu$ in Fact~\ref{fact:poly_approx_sqrt} within the runtime budget using the power method (Fact~\ref{fact:powermethod}), since $\kscale$ is given. This will only affect parameters in the remainder of the proof by constant factors. If $u = \mpack b$ for commuting $\mpack$ and $\mm$, our requirement is equivalent to
\[-\eps^2 \mm \preceq \Par{\mpack - \mm^\half}^2 \preceq \eps^2 \mm.\]
Since square roots are operator monotone (by the L\"owner-Heinz inequality), this is true iff
\[-\eps \mm^\half \preceq \mpack - \mm^\half \preceq \eps \mm^\half, \]
and such a $\mpack$ which is applicable within the runtime budget is given by Fact~\ref{fact:poly_approx_sqrt}.
\end{proof}

We next demonstrate two applications of Corollary~\ref{cor:implsqrt} in estimating applications of products involving $\ma$. We will use the following fact, whose proof is deferred to Appendix~\ref{app:kernel}.

\begin{restatable}{lemma}{restateopnormprod}\label{lem:opnormprod}
	Let $\mb \in \R^{d \times d}$ and let $\ma \in \PD^d$. Then $\min\Par{\norm{\ma\mb}_2, \norm{\mb\ma}_2} \ge \frac{1}{\kappa(\ma)}\norm{\mb}_2 \norm{\ma}_2$.
\end{restatable}

First, we discuss the application of a polynomial in $\ma \mw \ma$ to a random vector.

\begin{lemma}\label{lem:applyana}
Let $u \in \R^d$ be a uniformly random unit vector, let $\mk \in \PD^d$ such that $\ma \defeq \mk^\half$ and $\kappa(\mk) \le \kscale$, and let $\mpack$ be a degree-$\Delta$ polynomial in $\ma \mw \ma$ for some explicit diagonal matrix $\mw$. For $\delta, \eps \in (0, 1)$, with probability $\ge 1 - \delta$ we can compute $w \in \R^d$ so $\norm{w - \mpack u}_2 \le \eps \norm{\mpack u}_2$ in time
\[O\Par{\tmv(\mk) \cdot \Par{\Delta + \sqrt{\kscale} \log \frac{d\kscale}{\delta \eps}}}.\]
\end{lemma}
\begin{proof}
We can write $\mpack = \ma \mn \ma$ for some explicit matrix $\mn$ which is a degree-$O(\Delta)$ polynomial in $\mk$ and $\mw$, which we have explicit access to. Standard concentration bounds show that with probability at least $1 - \delta$, for some $N = \text{poly}(d, \delta^{-1})$, $\norm{\mpack u}_2 \ge \frac{1}{N}\norm{\mpack}_2$.
Condition on this event for the remainder of the proof, such that it suffices to obtain additive accuracy $\frac \eps N \norm{\mpack}_2$. By two applications of Lemma~\ref{lem:opnormprod}, we have
\begin{equation}\label{eq:anaop}
\norm{\ma \mn \ma}_2 \ge \frac 1 {\kscale} \norm{\ma}_2^2 \norm{\mn}_2.
\end{equation}
Our algorithm is as follows: for $\eps' \gets \frac{\eps}{3N\kscale}$, compute $v$ such that $\norm{v - \ma u}_2 \le \eps' \norm{\ma}_2 \norm{u}_2$ using Corollary~\ref{cor:implsqrt}, explicitly apply $\mn$, and then compute $w$ such that $\norm{w - \ma \mn v}_2 \le \eps' \norm{\ma}_2 \norm{\mn v}_2$; the runtime of this algorithm clearly fits in the runtime budget. The desired approximation is via
\begin{align*}
\norm{w - \ma \mn \ma u}_2 &\le \norm{w - \ma \mn v}_2 + \norm{\ma \mn v - \ma \mn \ma u}_2 \\
&\le \eps' \norm{\ma}_2 \norm{\mn}_2 \norm{v}_2 + \norm{\ma \mn v - \ma \mn \ma u}_2 \\
&\le 2\eps' \norm{\ma}_2^2 \norm{\mn}_2 + \norm{\ma}_2 \norm{\mn}_2 \norm{v - \ma u}_2 \\
&\le 2\eps' \norm{\ma}_2^2 \norm{\mn}_2 + \eps' \norm{\ma}_2^2 \norm{\mn}_2 \\
&\le 3\eps' \kscale \norm{\ma \mn \ma}_2 = \frac{\eps}{N}\norm{\mpack}_2.
\end{align*}
The third inequality used $\norm{v}_2 \le \norm{\ma u}_2 + \eps \norm{\ma}_2 \norm{u}_2 \le (1 + \eps)\norm{\ma}_2 \le 2\norm{\ma}_2$.
\end{proof}

We give a similar guarantee for random bilinear forms through $\ma$ involving an explicit vector.

\begin{lemma}\label{lem:applymaprod}
Let $u \in \R^d$ be a uniformly random unit vector, let $\mk \in \PD^d$ such that $\ma \defeq \mk^\half$ and $\kappa(\mk) \le \kscale$, and let $v \in \R^d$. For $\delta, \eps \in (0, 1)$, with probability $\ge 1 - \delta$ we can compute $w \in \R^d$ so $\inprod{w}{v}$ is an $\eps$-multiplicative approximation to $u^\top \ma v$ in time
\[O\Par{\tmv(\mk) \cdot \sqrt{\kscale}\log\frac{d\kscale}{\delta\eps}}.\]
\end{lemma}
\begin{proof}
As in Lemma~\ref{lem:applyana}, for some $N = \text{poly}(d, \delta^{-1})$ it suffices to give a $\frac{\eps}{N} \norm{\ma v}_2$-additive approximation. For $\eps' \gets \frac \eps {N\sqrt \kscale}$, we apply Corollary~\ref{cor:implsqrt} to obtain $w$ such that $\norm{w - \ma u}_2 \le \eps' \norm{\ma u}_2$, which fits within the runtime budget. Correctness follows from
\[
\Abs{\inprod{\ma u - w}{v}} \le \norm{\ma u - w}_2 \norm{v}_2 \le \eps' \norm{\ma}_2 \norm{v}_2 \le \eps' \sqrt{\kscale} \norm{\ma v}_2 \le \frac{\eps}{N} \norm{\ma v}_2.
\]
\end{proof}

\subsubsection{Estimating the smallest eigenvalue implicitly}

We begin by discussing implicit implementation of Line 13 of Algorithm~\ref{alg:decidempc}. Our strategy is to combine the approach of Lemma~\ref{lem:approxlmin} (applying the power method to the negative exponential), with Lemma~\ref{lem:applyana} since the power method involves products through random vectors. 

\begin{lemma}\label{lem:sqrtlmin}
Given $R > 1$, $\delta \in (0, 1)$, $\mk \in \PD^d$ such that $\ma \defeq \mk^\half$ and $\kappa(\mk) \le \kscale$, and diagonal $\mw \in \PSD^d$ such that $\mm \defeq \ma \mw \ma \preceq \kappa\id$, we can compute a $R$-additive approximation to $\lam_{\min}(\mm)$ with probability $\ge 1 - \delta$ in time
\[O\Par{\tmv(\mk) \cdot \sqrt{\kappa + \kscale} \cdot R\log\frac{d\kscale}{\delta}}.\]
\end{lemma}
\begin{proof}
The proof of Lemma~\ref{lem:approxlmin} implies it suffices to compute a $0.2$-multiplicative approximation to the largest eigenvalue of $\mpack$, a degree-$\Delta = O(\sqrt{\kappa R} + R)$ polynomial in $\mm$. Moreover, letting $\Delta' = O(\log \frac{n}{\delta})$ be the degree given by Fact~\ref{fact:powermethod} with $\delta \gets \frac \delta 3$, the statement of the algorithm in Fact~\ref{fact:powermethod} shows it suffices to compute for a uniformly random unit vector $u$,
\begin{align*}
\norm{\mpack^\Delta u}_2 \text{ and } \norm{\mpack^{\Delta + 1} u}_2 \text{ to multiplicative accuracy } \frac 1 {30}.
\end{align*}
We demonstrate how to compute $\norm{\mpack^\Delta u}_2$ to this multiplicative accuracy with probability at least $1 - \frac \delta 3$; the computation of $\norm{\mpack^{\Delta + 1} u}_2$ is identical, and the failure probability follows from a union bound over these three random events. Since $\mpack^\Delta$ is a degree-$O(\Delta\Delta') = O(\sqrt{\kappa}R \log \frac d \delta)$ polynomial in $\ma \mw \ma$, the conclusion follows from Lemma~\ref{lem:applyana}.
\end{proof}

\subsubsection{Estimating inner products with a negative matrix exponential implicitly}

We next implement Line 6, giving variants of Lemmas~\ref{lem:trexpnegm} and~\ref{lem:line5left} tolerating error of Corollary~\ref{cor:implsqrt}.

\begin{lemma}\label{lem:sqrttrexp}
Given $R > 1$, $\delta \in (0, 1)$, $\mk \in \PD^d$ such that $\ma \defeq \mk^\half$ and $\kappa(\mk) \le \kscale$, and diagonal $\mw \in \PSD^d$ such that $\mm \defeq \ma \mw \ma \preceq \kappa R\id$ and $\lmin(\mm) \le R$, for sufficiently small constant $\eps$ we can compute an $\eps$-multiplicative approximation to $\Tr\exp(-\mm)$ with probability $\ge 1 - \delta$ in time
\[O\Par{\tmv(\mk) \cdot \sqrt{\kappa + \kscale} \cdot R\log^2 \frac {d\kscale} \delta}.\]
\end{lemma}
\begin{proof}
The proof of Lemma~\ref{lem:trexpnegm} shows it suffices to compute $k = O(\log \frac d \delta)$ times, an $\frac \eps 3\exp(-R)$-additive approximation to $u^\top \mpack u$, for uniformly random unit vector $u$ and $\mpack$, a degree-$\Delta = O(\sqrt{\kappa R} + R)$-polynomial in $\mm$ with $\norm{\mpack}_2 \le \norm{\exp(-\mm)}_2 + \frac \eps 3 \exp(-R) \le \frac 4 3 \exp(-R)$. Applying Lemma~\ref{lem:applyana} with $\eps \gets \frac \eps 4$ to compute $w$, an approximation to $\mpack u$, the approximation follows from
\begin{align*}
\Abs{\inprod{w}{u} - \inprod{\mpack u}{u}} \le \norm{w - \mpack u}_2 \le \frac \eps 4 \norm{\mpack}_2 \le \frac \eps 3 \exp(-R).
\end{align*}
The runtime follows from the cost of applying Lemma~\ref{lem:applyana} to all $k$ random unit vectors.
\end{proof}

\begin{lemma}\label{lem:sqrtinprod}
Given $c > 1$, $\delta \in (0, 1)$, $\mk \in \PD^d$ such that $\ma \defeq \mk^\half$ and $\kappa(\mk) \le \kscale$, and diagonal $\mw \in \PSD^d$ such that $\mm \defeq \ma \mw \ma \preceq \kappa\id$, for sufficiently small constant $\eps$ we can compute $(\eps, \exp(-c))$-approximations to all 
\[\Brace{\inprod{a_ia_i^\top}{\exp(-\mm)}}_{i \in [d]} \text{ where } \Brace{a_i}_{i \in [d]} \text{ are rows of } \ma \text{ and } \norm{a_i}_2 \le \rho \text{ for all } i \in [d],\]
with probability $\ge 1 - \delta$ in time
\[O\Par{\tmv(\mk) \cdot \sqrt{\kappa + \kscale} \cdot \Par{c + \log \rho} \log^2 \frac {d\kscale} \delta}.\]
\end{lemma}
\begin{proof}
The proof of Lemma~\ref{lem:line5left} implies it suffices to compute $k = O(\log \frac d \delta)$ times, for each $i \in [d]$, the quantity $\inprod{u}{\mpack a_i}$ to multiplicative error $\frac \eps 2$, for uniformly random unit vector $u$ and $\mpack$, a degree-$\Delta = O(\sqrt{\kappa (c +\log \rho)} +(c \log \rho))$-polynomial in $\mm$; we remark we gain an additive $\log \rho$ in the degree to account for the scale of $a_ia_i^\top$. Next, note that since $a_i = \ma e_i$ and $\mpack = \ma \mn \ma$ for $\mn$ an explicit degree-$O(\Delta)$ polynomial in $\mk$ and $\mw$, we have $\inprod{u}{\mpack a_i} = u^\top \ma \Par{\mn \mk e_i}$. We can approximate this by some $\inprod{w}{\mn \mk e_i}$ via Lemma~\ref{lem:applymaprod} to the desired accuracy. The runtime comes from applying Lemma~\ref{lem:applymaprod} $k$ times, multiplying each of the resulting vectors $w$ by $\mk \mn$ and stacking them to form a $k \times d$ matrix $\tmq$, and then computing all $\|\tmq e_i\|_2$ for $i \in [d]$.
\end{proof}

\subsubsection{Implementing a packing oracle implicitly}

Finally, we discuss implementation of Line 7 of Algorithm~\ref{alg:decidempc}. In particular, letting $\mk$ be a matrix with symmetric square root $\ma$ (with rows $a_i \defeq \ai$), for the matrices $\ma_i = a_i a_i^\top$ the requirement of Line 7 is a multiplicative approximation (and a witnessing reweighting) to the optimization problem
\[\max_{\substack{\sum_{i \in [d]} w_i \ma_i \preceq \id \\ w \in \R^d_{\ge 0}}} v^\top w.\]
Here, $v$ is explicitly given by an implementation of Line 6 of the algorithm, but we do not have $\{\ma_i\}_{i \in [d]}$ explicitly. To implement this step implicitly, we recall the approximation requirements of the solver of Proposition~\ref{prop:optv}, as stated in \cite{Jambulapati0T20}. We remark that the approximation tolerance is stated for the decision problem tester of \cite{Jambulapati0T20} (Proposition~\ref{prop:decisionpack}); once the tester is implicitly implemented, the same reduction as described in Appendix~\ref{app:slowoptimal} yields an analog to Proposition~\ref{prop:optv}.

\begin{corollary}[Approximation tolerance of Proposition~\ref{prop:optv}, Theorem 5, \cite{Jambulapati0T20}]\label{cor:approxpack}
Let $\eps$ be a sufficiently small constant. The runtime of Proposition~\ref{prop:optv} is due to $T = O(\log(d(\delta)^{-1})\log d \cdot \log\log\frac{\opt_+}{\opt_-})$ iterations, each of which requires $O(1)$ vector operations and $O(\eps)$-multiplicative approximations to
\begin{equation}\label{eq:approxpack}\Tr\Par{\mm^p},\; \Brace{\inprod{\ma_i}{\mm^{p - 1}}}_{i \in [d]} \text{ for } \mm \defeq \sum_{i \in [d]} w_i \ma_i \text{ for an explicitly given } w \in \R^d_{\ge 0},\end{equation}
where $p = O(\log d\cdot \eps^{-1})$ is an odd integer, and $S\id \preceq \mm \preceq R\id$, for $R = O(\log d \cdot \eps^{-1})$ and $S = \textup{poly}(\frac \eps {nd}, \kappa((\sum_{i \in [n]} \ma_i))^{-1})$.
\end{corollary}

We remark that the lower bound $S$ comes from the fact that the initial matrix of the \cite{Jambulapati0T20} solver is a bounded scaling of $\sum_{i \in [n]} \ma_i$, and the iterate matrices are monotone in Loewner order. We now demonstrate how to use Lemmas~\ref{lem:applyana} and~\ref{lem:applymaprod} to approximate all quantities in \eqref{eq:approxpack}. Throughout the following discussion, we specialize to the case where each $\ma_i = a_ia_i^\top$, so $\mm$ in \eqref{eq:approxpack} will always have the form $\mm = \ma \mw \ma$ for diagonal $\mw \in \PSD^d$. We also always have $n = d$.

\begin{lemma}\label{lem:packdenom}
Given $R > 1$, $\delta \in (0, 1)$, $\mk \in \PD^d$ such that $\ma \defeq \mk^\half$ and $\kappa(\mk) \le \kscale$, and diagonal $\mw \in \PSD^d$ such that $S \id \preceq \mm \defeq \ma \mw \ma \preceq R\id$ where $S = \textup{poly}((d\kscale)^{-1})$, for sufficiently small constant $\eps$ we can compute an $\eps$-multiplicative approximation to $\Tr (\mm^p)$ for integer $p$ in time
\[O\Par{\tmv(\mk) \cdot \Par{p + \sqrt{\kscale}\log\frac{d\kscale}{\delta}} \cdot \log \frac d \delta}.\]
\end{lemma}
\begin{proof}
As in Lemma~\ref{lem:sqrttrexp}, it suffices to compute $k = O(\log \frac d \delta)$ times, an $\frac \eps N S^p$-additive approximation to $u^\top \mm^p u$, for uniformly random unit vector $u$ and $N = \text{poly}(d, \delta^{-1})$. By applying Lemma~\ref{lem:applyana} with accuracy $\eps' \gets \frac{\eps S^p}{N R^p}$ to obtain $w$, an approximation to $\mm^p u$, we have the desired
\[\Abs{\inprod{u}{\mm^p u} - \inprod{u}{w}} \le \norm{\mm^p u - w}_2 \le \eps' \norm{\mm^p u}_2 \le \eps' R^p \le \frac{\eps}{N} S^p.\]
The runtime follows from $k$ applications of Lemma~\ref{lem:applyana} to the specified accuracy level.
\end{proof}

\begin{lemma}\label{lem:packnum}
Given $R > 1$, $\delta \in (0, 1)$, $\mk \in \PD^d$ such that $\ma \defeq \mk^\half$ and $\kscale(\mk) \le \kscale$, and diagonal $\mw \in \PSD^d$ such that $\mm \defeq \ma \mw \ma \preceq R\id$, for sufficiently small constant $\eps$ we can compute an $\eps$-multiplicative approximation to all
\[\Brace{\inprod{a_ia_i^\top}{\mm^{p - 1}}}_{i \in [d]} \text{ where } \{a_i\}_{i \in [d]} \text{ are rows of } \ma,\]
where $p$ is an odd integer, with probability $\ge 1 - \delta$ in time
\[O\Par{\tmv(\mk) \cdot \Par{p + \sqrt{\kscale} \log \frac{d\kscale}{\delta}} \cdot \log\frac d \delta}.\]
\end{lemma}
\begin{proof}
First, observe that for all $i \in [d]$ it is the case that
\[\inprod{a_ia_i^\top}{\mm^{p - 1}} = \Par{\ma e_i}^\top \mm^{p - 1} \Par{\ma e_i} \ge S^{p - 1} \norm{\ma}_{2}^2 \kscale^{-2}.\]
Letting $r = \half(p - 1)$ and following Lemma~\ref{lem:sqrtinprod} and the above calculation, it suffices to show how to compute $k = O(\log \frac d \delta)$ times, for each $i \in [d]$, the quantity $\inprod{u}{\mm^{r} a_i}$ to multiplicative error $\frac \eps 2$, for uniformly random unit vector $u$ and $N = \text{poly}(d, \delta^{-1})$. As in Lemma~\ref{lem:sqrtinprod}, each such inner product is $u^\top \ma (\mn \mk e_i)$ for $\mn$ an explicit degree-$O(p)$ polynomial in $\mk$ and $\mw$. The runtime follows from applying Lemma~\ref{lem:applymaprod} $k$ times and following the runtime analysis of Lemma~\ref{lem:sqrtinprod}.
\end{proof}

\subsection{Homotopy method}\label{ssec:homotopy}

In this section, we use the tools of Sections~\ref{ssec:implsqrt} in the context of Algorithm~\ref{alg:decidempc} to design a homotopy method for rescaling PD matrices. We first make two helper observations, whose proofs are deferred to Appendix~\ref{app:kernel}. The first shows it is trivial to compute a near-optimal scaling for $\mk + \lam \id$ when $\lam$ is sufficiently large, and small enough $\lam$ suffices to solve the original problem in $\mk$. 

\begin{restatable}{lemma}{restatelambounds}\label{lem:lambounds}
Let $\mk \in \PD^d$. Then, $\kappa(\mk + \frac{1}{\eps} \lmax(\mk)\id) \le 1 + \eps$. Moreover, given a diagonal $\mw \in \PSD^d$ such that $\kappa(\mw^\half(\mk + \lam \id)\mw^\half) \le \kscale$ for $0 \le \lam \le \frac {\eps \lmin(\mk)}{1 + \eps}$, $\kappa(\mw^\half \mk \mw^\half) \le (1 + \eps)\kscale$.
\end{restatable}

The second demonstrates that given a near-optimal preconditioner for $\mk + \lam\id$, applying the same preconditioner to $\mk + \half \lam \id$ yields a condition number at most twice as large.

\begin{restatable}{lemma}{restatebetweenphase}\label{lem:betweenphase}
Let $\mk \in \PD^d$, and let $\mw \in \PSD^d$ be diagonal. Then for any $\lam > 0$,
\[\kappa\Par{\mw^\half \Par{\mk + \lam \id}\mw^\half} \le 2 \kappa\Par{\mw^\half \Par{\mk +\frac \lam 2 \id}\mw^\half}.\]
\end{restatable}

Finally, we give our main result for approximating $\ksk(\mk)$.

\begin{theorem}\label{thm:slowsqrt}
There is an algorithm, which given $\mk \in \PD^d$ computes $w \in \R^d_{\ge 0}$ such that $\kappa(\mw^\half \mk \mw^\half) = (1 + \eps)\ksk(\ma)$ for arbitrarily small $\eps = \Theta(1)$, with probability $\ge 1- \delta$ in time 
\[O\Par{\tmv(\mk) \cdot \Par{\ksk(\mk)}^{1.5} \cdot \log^{3}\Par{\frac{d\ksk(\mk)}{\delta}}\log^2(d)\log\Par{\ksk(\mk)}\log\log\Par{\ksk(\mk)} }.\]
\end{theorem}
\begin{proof}
Throughout this discussion, fix a sufficiently small constant $\eps$. First, given a value $\kappa \ge \frac{1}{1 - \eps} \ksk(\mk)$, we will discuss how to compute a rescaling $\mw$ such that $\kappa(\mw^\half \mk \mw^\half) \le (1 + O(\eps)) \kappa$. We then discuss how to incrementally search for the optimal value of $\kappa$ up to a constant factor. Finally, we assume that $\mk$ has been rescaled at the beginning of the algorithm so that its diagonal entries are all one, which by Proposition~\ref{prop:diagonesub} implies $\kappa(\mk) = O(\kappa^2)$.

\paragraph{Reduction to phases.} Our algorithm for searching for a rescaling with condition number $\le \kappa$ proceeds in $K = O(\log \kappa)$ phases. In particular, let $\lam_0$ be a $2$-approximation to $\frac 2 \eps \lam_{\max}(\mk)$, which we can compute with failure probability $\frac \delta 2$ within the given runtime budget using Fact~\ref{fact:powermethod}. Then, each phase $k \in [K]$ will compute a reweighting $\mw_k$ such that for $\lam_k \defeq \frac{\lam_0}{2^k}$,
\[\kappa\Par{\mw_k^\half\Par{\mk + \lam_k \id} \mw_k^\half} \le (1 + \eps) \kappa.\]
Lemma~\ref{lem:idcanthurt} implies this is always feasible. Given such a rescaling for phase $K$, Lemma~\ref{lem:lambounds} implies that such a rescaling is (up to constant factors in the $\eps$) also sufficient for rescaling $\mk$. Moreover, Lemma~\ref{lem:lambounds} also implies the identity rescaling $\mw_1 = \id$ suffices for the first phase. Finally, Lemma~\ref{lem:betweenphase} implies if we have succeeded in the $k^{\text{th}}$ phase for any $1 \le k \le K - 1$, we have access to a scaling with condition number $2(1 + \eps)\kappa \le 3 \kappa$ for the $(k + 1)^{\text{st}}$ phase. Hence, in phase $k$, redefining
\begin{equation}\label{eq:mkdefk}\mk \gets \mw_k^\half (\mk + \lam_k \id) \mw_k^\half,\end{equation}
we have reduced to the following self-contained problem for all phases $2 \le k \le K$: given $\kappa \ge 0$, $\eps \in (0, 1)$ such that $\ksk(\mk) \le (1 - \eps) \kappa$ and
 \begin{equation}\label{eq:mkprettygood}\kappa\Par{\mk} \le \kscale \defeq 3\kappa,\end{equation}
find a rescaling $\mw$ such that
\begin{equation}\label{eq:kappabound}\kappa\Par{\mw^\half \mk \mw^\half} \le (1 + \eps)\kappa.\end{equation}
Finally, we remark that in the $k^{\text{th}}$ phase and overloading $\mk$ to be defined as in \eqref{eq:mkdefk}, letting $\mk = \ma^2$ for $\ma \in \PD^d$, all rows $\{a_i\}_{i \in [d]}$ have $\norm{a_i}_2 \in [1, \rho]$ for $\rho \defeq O(\kappa^{0.5})$ without loss of generality. This follows from $\norm{a_i}_2^2 = \mk_{ii}$, and since by assumption all eigenvalues of $\mk$ are in a multiplicative range of $\kscale$ (and hence so are all diagonal entries), so by scale invariance we assume they are in $[1, \rho^2]$.

\paragraph{Implementing a single phase.} We now implement a single phase, which solves the problem of computing a rescaling $\mw$ satisfying \eqref{eq:kappabound}, under the initial guarantee \eqref{eq:mkprettygood}. Per the above discussion, we assume all rows $\{a_i\}_{i \in [d]}$ of $\ma \defeq \mk^\half$ have $\ell_2$ norms in the range $[1, \rho \defeq O(\kappa^{0.5})]$.

Next, to compute the reweighting $\mw$ satisfying \eqref{eq:kappabound} we run Algorithm~\ref{alg:decidempc} on the matrices $\{\ma_i \defeq a_ia_i^\top\}_{i \in [d]}$. We note that implementing all of Lines 6, 7, and 13 of Algorithm~\ref{alg:decidempc} (the runtime bottleneck) are doable with access to only $\mk$ and a currently maintained reweighting, in time
\[O\Par{\tmv(\mk) \cdot \sqrt{\kappa} \cdot \log^{3}\Par{\frac{d\kappa}{\delta}}\log(d)\log\log\Par{\kappa} }.\]
To achieve this, we use the bounds derived in Section~\ref{ssec:implsqrt} with the following parameters.

\emph{Line 6.} First, we estimate the denominator of each $\my_t$ in Algorithm~\ref{alg:decidempc} using Lemma~\ref{lem:sqrttrexp} with $\kappa \gets O(\kappa)$, $\kscale \gets \kscale$, and $R \gets O(\log d)$. Next, we estimate all numerators of inner products through $\my_t$ using Lemma~\ref{lem:sqrtinprod} with $\kappa \gets O(\kappa \log d)$, $\kscale \gets \kscale$, $c \gets O(\log (\kappa d))$, and $\rho \gets O(\kappa^{0.5})$, which suffices after rescaling by the denominator (as demonstrated in Lemma~\ref{lem:alg1left}). 

\emph{Line 7.} By Corollary~\ref{cor:approxpack}, it suffices to give implementations of $O(\eps)$-multiplicative approximations to the quantities in \eqref{eq:approxpack}. We use Lemma~\ref{lem:packdenom} with $\kscale \gets \kscale$ and $p \gets O(\log d)$ to approximate the denominator, and then Lemma~\ref{lem:packnum} with the same parameters to approximate the numerators. There is an overhead of two logarithmic factors due to the iteration bound of Corollary~\ref{cor:approxpack}.

\emph{Line 13.} We use Lemma~\ref{lem:sqrtlmin} with $R \gets O(\log d)$, $\kscale \gets \kscale$, and $\kappa \gets O(\kappa \log d)$. 

Altogether, Lines 6-13 require at most $O(\kappa \log d)$ calls to these respective procedures, so to obtain the failure probability it suffices to set $\delta \gets \frac{\delta}{O(\kappa \log d)}$ in their statements and take a union bound. By combining with the iteration bound of Algorithm~\ref{alg:decidempc} the overall cost of a phase is
\[O\Par{\tmv(\mk)\cdot \kappa^{1.5} \cdot \log^{3}\Par{\frac {d\kappa} \delta} \log^2(d)\log\log\Par{\kappa} }.\]

\paragraph{Cleaning up.} Finally, it is clear the cost of implementing all $K = O(\log \kappa)$ phases is the same as the above display with a multiplicative overhead of $K$. Here, we use that for any $\lam$, $\tmv(\mk + \lam \id) = O(1) \tmv(\mk)$. It remains to perform an incremental search for the value of $\kappa$; to do so, we follow the strategy of Theorem~\ref{thm:leftslow}, initalizing at $\kappa = 1$ and incrementing by factors of $1 + O(\eps)$. We will never run the algorithm with a value $\kappa > (1 + O(\eps))\ksk(\mk)$, and whenever any run of Algorithm~\ref{alg:decidempc} fails, we can safely conclude the current $\kappa$ is too small. If no run of Algorithm~\ref{alg:decidempc} fails in a given run, then we successfully compute the desired reweighting. As discussed in the proof of Theorem~\ref{thm:leftslow}, this incremental search only adds a constant overhead multiplicatively to the runtime.
\end{proof} 	%

\section{Faster constant-factor optimal scalings with a conjectured subroutine}
\label{sec:conjoptimal}

In this section, we demonstrate algorithms which achieve runtimes which scale as $\tO(\sqrt{\kappa^\star})$\footnote{Throughout this section for brevity, we use $\kappa^\star$ to interchangeably refer to the quantities $\ksl$ or $\ksr$ of a particular appropriate inner or outer rescaling problem.} matrix-vector multiplies for computing approximately optimal scalings, assuming the existence of a sufficiently general \emph{width-independent} mixed packing and covering (MPC) SDP solver. Such runtimes (which improve each of  Theorems~\ref{thm:leftslow},~\ref{thm:rightslow}, and~\ref{thm:slowsqrt} by roughly a $\kappa^\star$ factor) would nearly match the cost of the fastest solvers \emph{after} rescaling, e.g.\ conjugate gradient methods. We also demonstrate that we can achieve near-optimal algorithms for computing constant-factor optimal scalings for average-case notions of conditioning under this assumption. 

We first recall the definition of the general MPC SDP feasibility problem.

\begin{definition}[MPC feasibility problem]\label{def:mpc}
	Given sets of matrices $\{\mpack_i\}_{i \in [n]} \in \PSD^{d_p}$ and $\{\mcov_i\}_{i \in [n]} \in \PSD^{d_c}$, and error tolerance $\eps \in (0, 1)$, the mixed packing-covering (MPC) feasibility problem asks to return weights $w \in \R^n_{\ge 0}$ such that
	\begin{equation}\label{eq:feasible} \lmax\Par{\sum_{i \in [n]} w_i \mpack_i} \le (1 + \eps)\lmin\Par{\sum_{i \in [n]} w_i \mcov_i},\end{equation}
	or conclude that the following is infeasible for $w \in \R^n_{\ge 0}$:
	\begin{equation}\label{eq:infeasible}  \lmax\Par{\sum_{i \in [n]} w_i \mpack_i} \le \lmin\Par{\sum_{i \in [n]} w_i \mcov_i}.\end{equation}
	If both \eqref{eq:feasible} is feasible and \eqref{eq:infeasible} is infeasible, either answer is acceptable.
	\end{definition}
Throughout this section, we provide efficient algorithms under Assumption~\ref{assume:mpc}: namely, that there exists a solver for the MPC feasibility problem at constant $\eps$ with polylogarithmic iteration complexity and sufficient approximation tolerance. Such a solver would improve upon our algorithm in Section~\ref{ssec:mpcslow} both in generality (i.e.\ without the restriction that the constraint matrices are rank-one and multiples of each other) and in the number of iterations.

\begin{assumption}
	\label{assume:mpc}
	There is an algorithm $\MPC$ which takes inputs $\{\mpack_i\}_{i \in [n]} \in \PSD^{d_p}$, $\{\mcov_i\}_{i \in [n]} \in \PSD^{d_c}$, and error tolerance $\eps$, and solves problem \eqref{eq:feasible}, \eqref{eq:infeasible}, in $\poly(\log (n d \rho), \eps^{-1})$ iterations, where $d \defeq \max(d_p, d_c)$, $\rho \defeq \max_{i \in [n]} \tfrac{\lmax(\mcov_i)}{\lmax(\mpack_i)}$.
	Each iteration uses $O(1)$ $n$-dimensional vector operations, and for $\eps' = \Theta(\eps)$ with an appropriate constant, additionally requires computation of
	\begin{equation}\label{eq:gradients}\begin{aligned}
	\eps'\text{-multiplicative approximations to  }\inprod{\mpack_i}{\frac{\exp\Par{\sum_{i \in [n]} w_i \mpack_i}}{\Tr\exp\Par{\sum_{i \in [n]} w_i \mpack_i}}}\; \forall i\in[n], \\
	\Par{\eps', e^{\frac{-\log(nd\rho)}{\eps'}}\Tr(\mcov_i)}\text{-approximations to }\inprod{\mcov_i}{\frac{\exp\Par{-\sum_{i \in [n]} w_i \mcov_i}}{\Tr\exp\Par{-\sum_{i \in [n]} w_i \mcov_i}}}\; \forall i\in[n],
	\end{aligned}\end{equation}
	for $w \in \R^n_{\ge 0}$ with $\lmax\Par{\sum_{i \in [n]} w_i \mpack_i}, \lmin\Par{\sum_{i \in [n]} w_i \mcov_i} \le R$ for $R = O(\tfrac{\log (nd\rho)}{\eps})$.
\end{assumption}
In particular, we observe that the number of iterations of this conjectured subroutine depends polylogarithmically on $\rho$, i.e.\ the runtime is \emph{width-independent}.\footnote{The literature on approximate solvers for positive linear programs and semidefinite programs refer to logarithmic dependences on $\rho$ as width-independent, and we follow this convention in our exposition.} In our settings computing optimal rescaled condition numbers, $\rho = \Theta(\kappa^\star)$; our solver in Section~\ref{ssec:mpcslow} has an iteration count depending linearly on $\rho$. Such runtimes are known for MPC linear programs~\cite{mahoney2016approximating}, however, such rates have been elusive in the SDP setting. While the form of requirements in~\eqref{eq:gradients} may seem somewhat unnatural at first glance, we observe that this is the natural generalization of the error tolerance of known width-independent MPC LP solvers~\cite{mahoney2016approximating}. Moreover, these approximations mirror the tolerances of our width-dependent solver in Section~\ref{ssec:mpcslow} (see Line 6 and Corollary~\ref{cor:approxpack}).

We first record the following technical lemma, which we will repeatedly use.
\begin{lemma}\label{lem:line5leftold}
	Given a matrix $\mzero \preceq \mm \preceq R\id$ for some $R > 0$, sufficiently small constant $\eps$, and $\delta \in (0, 1)$, we can compute $\eps$-multiplicative approximations to the quantities 
	\[\inprod{a_ia_i^\top}{\exp(\mm)} \text{ for all $i \in [n]$, and } \Tr\exp(\mm)\]
	in time $O((\tmv(\mm) R + \nnz(\ma)) \log \frac n \delta)$, with probability at least $1 - \delta$.
\end{lemma}
\begin{proof}
	We discuss both parts separately. Regarding computing the inner products, equivalently, the goal is to compute approximations to all $\norm{\exp(\half \mm) a_i}_2^2$ for $i \in [n]$. First, by an application of Fact~\ref{fact:polyexp} with $\delta = \frac{\eps} 8 \exp(-2R)$, and then multiplying all sides of the inequality by $\exp(R)$, there is a degree-$O(R)$ polynomial such that
	\begin{gather*}
	\Par{1 - \frac{\eps}{8}}\exp\Par{\half \mm} \preceq \exp\Par{\half \mm} - \frac{\eps}{8} \id \preceq p\Par{\half \mm} \preceq \exp\Par{\half \mm} + \frac{\eps}{8}\id \preceq \Par{1 + \frac{\eps}{8}}\exp\Par{\half \mm} \\
	\implies \Par{1 - \frac{\eps}{3}} \exp(\mm) \preceq p\Par{\half \mm}^2 \preceq \Par{1 + \frac{\eps}{3}} \exp(\mm).
	\end{gather*}
	This implies that $\norm{p(\half \mm) a_i}_2^2$ approximates $\norm{\exp(\half \mm)a_i}_2^2$ to a multiplicative $\frac{\eps}{3}$ by the definition of Loewner order. Moreover, applying Fact~\ref{fact:jl} with a sufficiently large $k = O(\log \frac n \delta)$ implies by a union bound that for all $i \in [n]$, $\norm{\mq p(\half \mm) a_i}_2^2$ is a $\eps$-multiplicative approximation to $\norm{\exp(\half \mm)a_i}_2^2$. To compute all the vectors $\mq p(\half \mm) a_i$, it suffices to first apply $p(\half \mm)$ to all rows of $\mq$, which takes time $O(\tmv(\mm) \cdot kR)$ since $p$ is a degree-$O(R)$ polynomial. Next, once we have the explicit $k \times d$ matrix $\mq p(\half \mm)$, we can apply it to all $\{a_i\}_{i \in [n]}$ in time $O(\nnz(\ma) \cdot k)$.
	
	Next, consider computing $\Tr\exp(\mm)$, which by definition has
	\[\Tr\exp(\mm) = \sum_{j \in [d]} \norm{\Brack{\exp\Par{\half \mm}}_{j:}}_2^2.\]
	Applying the same $\mq$ and $p$ as before, we have by the following sequence of equalities
	\begin{align*}
	\sum_{j \in [d]} \norm{\mq \Brack{\exp\Par{\half \mm}}_{j:}}_2^2 &= \Tr\Par{\exp\Par{\half \mm} \mq^\top \mq \exp\Par{\half \mm}} \\
	&= \Tr\Par{\mq \exp\Par{\mm} \mq^\top} = \sum_{\ell \in [k]} \norm{\exp\Par{\half \mm} \mq_{\ell:}}_2^2,
	\end{align*}
	that for the desired approximation, it instead suffices to compute
	\[\sum_{\ell \in [k]} \norm{p\Par{\half \mm} \mq_{\ell:}}_2^2.\]
	This can be performed in time $O(\tmv(\mm) \cdot kR)$ as previously argued.
\end{proof}
A straightforward modification of this proof alongside Lemma~\ref{lem:applyana} also implies that we can compute these same quantities to $p(\ma \mw \ma)$, when we are only given $\mk = \ma^2$, assuming that $\mk$ is reasonably well-conditioned.
We omit the proof, as it follows almost identically to the proofs of Lemmas~\ref{lem:line5leftold},~\ref{lem:sqrttrexp}, and~\ref{lem:sqrtinprod}, the latter two demonstrating how to appropriately apply Lemma~\ref{lem:applyana}.

\begin{corollary}
	\label{cor:line5leftoldker}
	Let $\mk \in \PD^d$ such that $\mk = \ma^2$ and $\kappa (\mk) \leq \kscale$.
	Let $\mw$ be a diagonal matrix such that $\lmax (\ma \mw \ma) \leq R$.
	For $\delta, \eps \in (0, 1)$, we can compute $\eps$-multiplicative approximations to 
	\[
		\inprod{a_ia_i^\top}{\exp(\ma \mw \ma)} \text{ for all $i \in [n]$, and } \Tr\exp(\ma \mw \ma)
	\]
	with probability $\ge 1 - \delta$ in time $O\Par{\tmv(\mk) \cdot R \cdot \Par{R + \sqrt{\kscale} \log \frac{d \kscale}{\delta}} \log \frac n \delta}$.
\end{corollary}

\subsection{Approximating $\kappa^\star$ under Assumption~\ref{assume:mpc}}\label{ssec:ksfast}

In this section, we show that, given Assumption~\ref{assume:mpc}, we obtain improved runtimes for all three types of diagonal scaling problems, roughly improving Theorems~\ref{thm:leftslow},~\ref{thm:rightslow}, and~\ref{thm:slowsqrt} by a $\kappa^\star$ factor.

\paragraph{Inner scalings.}
We first demonstrate this improvement for inner scalings.
\begin{theorem}
	\label{thm:conjleft}
	Under Assumption~\ref{assume:mpc}, there is an algorithm which, given full-rank $\ma \in \R^{n \times d}$ for $n \geq d$ computes $w \in \R^n_{\ge 0}$ such that $\kappa(\ma^\top \mw \ma) \le (1 + \eps)\ksl(\ma)$ for arbitrarily small $\eps = \Theta(1)$, with probability $\ge 1 - \delta$ in time
	\[
		O\left(\nnz (\ma) \cdot \sqrt{\ksl(\ma)} \cdot \poly \log \frac{n \ksl(\ma)}{ \delta}  \right) \; .
	\]
\end{theorem}
\begin{proof}
	For now, assume we know $\ksl(\ma)$ exactly, which we denote as $\ksl$ for brevity.
	Let $\{a_i\}_{i \in [n]}$ denote the rows of $\ma$, and assume that $\norm{a_i}_2 = 1$ for all $i \in [n]$.
	By scale invariance, this assumption is without loss of generality.
	We instantiate Assumption~\ref{assume:mpc} with $\mpack_i = a_i a_i^\top$ and $\mcov_i = \ksl a_i a_i^\top$, for $i \in [n]$.
	It is immediate that a solution yields an inner scaling with the same quality up to a $1 + \eps$ factor, because by assumption \eqref{eq:infeasible} is feasible so $\MPC$ cannot return ``infeasible.''
	
	We now instantiate the primitives in~\eqref{eq:gradients} needed by Assumption~\ref{assume:mpc}. Throughout, note that $\rho = \ksl$ in this setting. Since we run $\MPC$ for $\poly (\log n\ksl)$ iterations, we will set $\delta' \gets \delta \cdot (\poly(n\ksl))^{-1}$ for the failure probability of each of our computations in \eqref{eq:gradients}, such that by a union bound all of these computations are correct.
	
	By Lemma~\ref{lem:line5leftold}, we can instantiate the packing gradients to the desired approximation quality in time $O(\nnz (\ma) \cdot \poly\log\tfrac{n\ksl}{\delta})$ with probability $1 - \delta'$. By Lemmas~\ref{lem:trexpnegm} and~\ref{lem:line5left}, we can instantiate the covering gradients in time $O(\nnz (\ma) \sqrt{\ksl} \cdot \poly\log\frac{n\ksl}{\delta})$ 
	with probability $1 - \delta'$. In applying these lemmas, we use the assumption that $\lmin(\sum_{i \in [n]} w_i \mcov_i) = O(\log n\ksl)$ as in Assumption~\ref{assume:mpc}, and that the covering matrices are a $\ksl$ multiple of the packing matrices so $\lmax(\sum_{i \in [n]} w_i \mcov_i) = O(\ksl \log n \ksl)$. Thus, the overall runtime of all iterations is 
	\[
		O\Par{\nnz(\ma) \cdot \sqrt{\ksl} \cdot \poly \log \frac{n \ksl}{ \delta}  }
	\]
	for $\eps = \Theta (1)$.
	To remove the assumption that we know $\ksl(\ma)$, we can use an incremental search (see Theorem~\ref{thm:leftslow}) on the scaling multiple between $\{\mcov_i\}_{i \in [n]}$ and $\{\mpack_i\}_{i \in [n]}$, starting from $1$ and increasing by factors of $1 + \eps$, adding a constant overhead to the runtime. Our width will never be larger than $O(\ksl(\ma))$ in any run, since $\MPC$ must conclude feasible when the width is sufficiently large.
\end{proof}

\paragraph{Outer scalings.} For simplicity, we will only discuss the case where we wish to symmetrically outer scale a matrix $\mk \in \PD^d$ near-optimally (i.e.\ demonstrating an improvement to Theorem~\ref{thm:slowsqrt} under Assumption~\ref{assume:mpc}). In the case where we have a factorization $\mk = \ma^\top \ma$, a similar improvement to Theorem~\ref{thm:rightslow} immediately follows since $\tmv(\mk) = O(\nnz(\ma))$, so we omit this discussion.

\begin{theorem}\label{thm:conjright}
Under Assumption~\ref{assume:mpc}, there is an algorithm which, given $\mk \in \PD^d$ computes $w \in \R^d_{\ge 0}$ such that $\kappa(\mw^\half \mk \mw^\half) \le (1 + \eps) \ksk(\mk)$ for arbitrarily small $\eps \in \Theta(1)$, with probability $\ge 1 - \delta$ in time
\[O\Par{\tmv(\mk) \cdot \sqrt{\ksk(\mk)} \cdot \poly\log \frac{d\ksk(\mk)}{\delta}}.\]
\end{theorem}
\begin{proof}
Throughout we denote $\ksk \defeq \ksk(\mk)$ for brevity. Our proof follows that of Theorem~\ref{thm:slowsqrt}, which demonstrates that it suffices to reduce to the case where we have a $\mk \in \PD^d$ with $\kappa(\mk) \le \kscale \defeq 3\ksk$, and we wish to find an outer diagonal scaling $\mw \in \PD^d$ such that $\kappa(\mw^\half \mk \mw^\half) \le (1 + \eps) \ksk$. We incur a polylogarithmic overhead on the runtime of this subproblem, by using it to solve all phases of the homotopy method in Theorem~\ref{thm:slowsqrt}, and the cost of an incremental search on $\ksk$.  

To solve this problem, we again instantiate Assumption~\ref{assume:mpc} with $\mpack_i = a_i a_i^\top$ and $\mcov_i = \ksk a_i a_i^\top$, where $\{a_i\}_{i \in [d]}$ are rows of $\ma \defeq \mk^\half$. As in Section~\ref{sec:kernel}, the main difficulty is to implement the gradients in~\eqref{eq:gradients} with only implicit access to $\ma$ , which we again will perform to probability $1 - \delta'$ for some $\delta' = \delta \cdot (\poly(n\ksk))^{-1}$ which suffices by a union bound.
Applying Lemmas~\ref{lem:sqrttrexp} and~\ref{lem:sqrtinprod} with the same parameters as in the proof of Theorem~\ref{thm:slowsqrt} (up to constants) implies that we can approximate the covering gradients in \eqref{eq:gradients} to the desired quality within time
\[
O \Par{\tmv (\mk) \cdot \sqrt{\ksk} \cdot \poly \log \frac{n\ksk}{\delta} } .	
\]
Similarly, Corollary~\ref{cor:line5leftoldker} implies we can compute the necessary approximate packing gradients in the same time. Multiplying by the overhead of the homotopy method in Theorem~\ref{thm:slowsqrt} gives the result.
\end{proof}

\subsection{Average-case conditioning under Assumption~\ref{assume:mpc}}

A number of recent linear system solvers depend on average notions of conditioning, namely the ratio between the average eigenvalue and smallest \cite{StrohmerV06, LeeS13, Johnson013, DefazioBL14, ZhuQRY16, Allen-Zhu17, AgarwalKKLNS20}. Normalized by dimension, we define this average conditioning as follows: for $\mm \in \PD^d$,
\[\tau\Par{\mm} \defeq \frac{\Tr(\mm)}{\lmin(\mm)}.\]
Observe that since $\Tr(\mm)$ is the sum of eigenvalues, the following inequalities always hold:
\begin{equation}\label{eq:taubounds}
d \le \tau\Par{\mm} \le d\kappa\Par{\mm}.
\end{equation}
In analogy with $\ksl$ and $\ksr$, we define for full-rank $\ma \in \R^{n \times d}$ for $n \ge d$, and $\mk \in \PD^d$,
\begin{equation}\label{eq:taustardef}\tsl(\ma) \defeq \min_{\textup{diagonal } \mw \succeq \mzero} \tau\Par{\ma^\top \mw \ma},\; \tsr(\mk) \defeq \min_{\textup{diagonal } \mw \succeq \mzero} \tau\Par{\mw^\half \mk \mw^\half}. \end{equation}
We give an informal discussion on how to use Assumption~\ref{assume:mpc} to develop a solver for approximating $\tsl$ to a constant factor, which has a runtime nearly-matching the fastest linear system solvers depending on $\tsl$ after applying the appropriate rescalings.\footnote{We remark that these problems may be solved to high precision by casting them as an appropriate SDP and applying general SDP solvers, but in this section we focus on fast runtimes.} Qualitatively, this may be thought of as the average-case variant of Theorem~\ref{thm:conjleft}. We defer an analogous result on approximating $\tsr$ (with or without a factorization) to future work for brevity. We remark that a solver for Assumption~\ref{assume:mpc} which symmetrically weights an ``imagined'' orthogonal basis (following Section~\ref{ssec:ksrslow}) likely extends to apply to outer scalings with a factorization. 

To develop our algorithm for approximating $\tsl$, we require several tools. The first is the rational approximation analog of the polynomial approximation in Fact~\ref{fact:polyexp}.

\begin{fact}[Rational approximation of $\exp$ \cite{SachdevaV14}, Theorem 7.1]\label{fact:ratexp}
	Let $\mm \in \PSD^d$ and $\delta > 0$. There is an explicit polynomial $p$ of degree $\Delta = \Theta(\log(\delta^{-1}))$ with absolute coefficients at most $\Delta^{O(\Delta)}$ with
	\[\exp(-\mm) - \delta \id \preceq p\Par{\Par{\id + \frac{\mm}{\Delta}}^{-1}} \preceq \exp(-\mm) + \delta\id.\] 
\end{fact}
We also use the runtime of the fastest-known solver for linear systems based on row subsampling, with a runtime dependent on the average conditioning $\tau$. Our goal is to compute reweightings $\mw$ which approximately attain the minimums in \eqref{eq:taustardef}, with runtimes comparable to that of Fact~\ref{fact:sotasgd}.

\begin{fact}[\cite{AgarwalKKLNS20}]\label{fact:sotasgd}
There is an algorithm which given $\mm \in \PD^d$, $b \in \R^d$, and $\delta, \eps \in (0, 1)$ returns $v \in \R^d$ such that
$\norm{v - \mm^{-1} b}_2 \le \eps \norm{\mm^{-1} b}_2$ with probability $\ge 1 - \delta$ in time
\[O\Par{\Par{n + \sqrt{d \tau(\mm)}} \cdot d \cdot  \poly\log\frac{n\tau(\mk)}{\delta\eps}}.\]
\end{fact}
\noindent\textit{Remark.} The runtime of Fact~\ref{fact:sotasgd} applies more broadly to quadratic optimization problems in $\mm$, e.g.\ regression problems of the form $\norm{\ma x - b}_2^2$ where $\ma^\top \ma = \mm$. Moreover, Fact~\ref{fact:sotasgd} enjoys runtime improvements when the rows of $\mm$ (or the factorization component $\ma$) are sparse; our methods in the following discussion do as well as they are directly based on Fact~\ref{fact:sotasgd}, and we omit this discussion for simplicity. Finally, \cite{AgarwalKKLNS20} demonstrates how to improve the dependence on $\sqrt{d\tau(\mm)}$ to a more fine-grained quantity in the case of non-uniform eigenvalue distributions. We defer obtaining similar improvements for approximating optimal rescalings to interesting future work.

\medskip
We now give a sketch of how to use Facts~\ref{fact:ratexp} and~\ref{fact:sotasgd} to obtain near-optimal runtimes for computing a rescaling approximating $\tsl$ under Assumption~\ref{assume:mpc}. Let $\ma \in \R^{n \times d}$ for $n \ge d$ be full rank, and assume that we known $\tsl \defeq \tsl(\ma)$ for simplicity, which we can approximate using an incremental search with a logarithmic overhead. Denote the rows of $\ma$ by $\{a_i\}_{i \in [n]}$. We instantiate Assumption~\ref{assume:mpc} with 
\begin{equation}\label{eq:packcovtau}\mpack_i = \norm{a_i}_2^2,\; \mcov_i = \tsl a_ia_i^\top, \text{ for all } i \in [n],\end{equation}
from which it follows that \eqref{eq:infeasible} is feasible using the reweighting $\mw = \diag{w}$ attaining $\tsl$:
\begin{equation}\label{eq:correcttrace}
\begin{aligned}\lmax\Par{\sum_{i \in [n]} w_i \mpack_i} &= \lmax\Par{\sum_{i \in [n]} w_i \norm{a_i}_2^2} = \Tr\Par{\ma^\top \mw \ma}, \\
\lmin\Par{\sum_{i \in [n]} w_i \mcov_i} &= \tsl \lmin\Par{\sum_{i \in [n]} w_i a_ia_i^\top} = \tsl\lmin\Par{\ma^\top \mw \ma}.
\end{aligned}
\end{equation}
Hence, if we can efficiently implement each step of $\MPC$ with these matrices, it will return a reweighting satisfying \eqref{eq:feasible}, which yields a trace-to-bottom eigenvalue ratio approximating $\tsl$ to a $1 + \eps$ factor. We remark that in the algorithm parameterization, we have $\rho = \tsl$. Moreover, all of the packing gradient computations in \eqref{eq:gradients} are one-dimensional and hence amount to vector operations, so we will only discuss the computation of covering gradients.

Next, observe that Assumption~\ref{assume:mpc} guarantees that for all intermediate reweightings $\mw$ computed by the algorithm and $R = O(\log n\tsl)$, $\lmax(\sum_{i \in [n]} w_i \mpack_i) = \Tr(\ma^\top \mw \ma) \le R$. This implies that the trace of the matrix involved in covering gradient computations is always bounded:
\begin{equation}\label{eq:tracemm}\Tr\Par{\sum_{i \in [n]} w_i \mcov_i} = \tsl \Tr\Par{\ma^\top \mw \ma} \le \tsl R.\end{equation}
To implement the covering gradient computations, we appropriately modify Lemmas~\ref{lem:trexpnegm} and~\ref{lem:line5left} to use the rational approximation in Fact~\ref{fact:ratexp} instead of the polynomial approximation in Fact~\ref{fact:polyexp}. It is straightforward to check that the degree of the rational approximation required is $\Delta = O(\log n\tsl)$. 

Moreover, each of the $\Delta$ linear systems which Fact~\ref{fact:ratexp} requires us to solve is in the matrix
\[\mm \defeq \id + \frac{\sum_{i \in [n]} w_i \mcov_i}{\Delta},\]
which by \eqref{eq:tracemm} and the fact that $\id$ has all eigenvalues $1$, has $\tau(\mm) = O(\tsl)$. Thus, we can apply Fact~\ref{fact:sotasgd} to solve these linear systems in time
\[O\Par{\Par{n + \sqrt{d\tsl}} \cdot d \cdot \poly\log\frac{n\tsl}{\delta}}.\]
Here, we noted that the main fact that e.g.\  Lemmas~\ref{lem:trexpnegm} and~\ref{lem:line5left} use is that the rational approximation approximates the exponential up to a $\text{poly}(n^{-1}, (\tsl)^{-1})$ multiple of the identity. Since all coefficients of the polynomial in Fact~\ref{fact:ratexp} are bounded by $\Delta^{O(\Delta)}$, the precision to which we need to apply Fact~\ref{fact:sotasgd} to satisfy the requisite approximations is $\eps = \Delta^{-O(\Delta)}$, which only affects the runtime by polylogarithmic factors. Combining the cost of computing \eqref{eq:gradients} with the iteration bound of Assumption~\ref{assume:mpc}, the overall runtime of our method for approximating $\tsl$ is
\[O\Par{\Par{n + \sqrt{d\tsl(\ma)}} \cdot d \cdot \poly\log\frac{n\tsl(\ma)}{\delta}},\]
which matches Fact~\ref{fact:sotasgd}'s runtime after rescaling in all parameters up to logarithmic factors.

\section{Applications}
\label{sec:apps}

In this section, we give a number of applications of our rescaling methods to problems in statistical settings (i.e.\ linear system solving or statistical regression) where reducing conditioning measures are effective. We begin by discussing connections between diagonal preconditioning and a semi-random noise model for linear systems in Section~\ref{ssec:semirandom}. We then apply rescaling methods to reduce risk bounds for statistical models of linear regression in Section~\ref{ssec:regression}.

\subsection{Semi-random linear systems}\label{ssec:semirandom}

Consider the following semi-random noise model for solving an overdetermined, consistent linear system $\ma \xtrue = b$ where $\ma \in \R^{n \times d}$ for $n \ge d$. 

\begin{definition}[Semi-random linear systems]\label{def:semirandom}
In the \emph{semi-random noise model} for linear systems, a matrix $\ma_g \in \R^{m \times d}$ with $\kappa(\ma_g^\top \ma_g) = \kappa_g$, $m \ge d$ is ``planted'' as a subset of rows of a larger matrix $\ma \in \R^{n \times d}$. We observe the vector $b = \ma \xtrue$ for some $\xtrue \in \R^d$ we wish to recover.
\end{definition} 

We remark that we call the model in Definition~\ref{def:semirandom} ``semi-random'' because of the following motivating example: the rows $\ma_g$ are feature vectors drawn from some ``nice'' (e.g.\ well-conditioned) distribution, and the dataset is contaminated by an adversary supplying additional data (a priori indistinguishable from the ``nice'' data), aiming to hinder conditioning of the resulting system. 

Interestingly, Definition~\ref{def:semirandom} demonstrates in some sense a shortcoming of existing linear system solvers: their brittleness to \emph{additional, consistent information}. In particular, $\kappa(\ma^\top \ma)$ can be arbitrarily larger than $\kappa_g$. However, if we were given the indices of the subset of rows $\ma_g$, we could instead solve the linear system $b_g = \ma_g \xtrue$ with iteration count dependent on the condition number of $\ma_g$. Counterintuitively, by giving additional rows, the adversary can arbitrarily increase the condition number of the linear system, hindering the runtime of conditioning-dependent solvers. 

The inner rescaling algorithms we develop in Sections~\ref{ssec:kslslow} and~\ref{ssec:ksfast} are well-suited for robustifying linear system solvers to the type of adversary in Definition~\ref{def:semirandom}. In particular, note that
\[\ksl\Par{\ma} \le \kappa\Par{\ma^\top \mw_g \ma} = \kappa\Par{\ma_g^\top\ma_g} = \kappa_g,\]
where $\mw_g$ is the diagonal matrix which is the $0$-$1$ indicator of rows of $\ma_g$. Our solvers for reweightings approximating $\ksl$ can thus be seen as trading off the \emph{sparsity} of $\ma_g$ for the potential of ``mixing rows'' to attain a runtime dependence on $\ksl(\ma) \le \kappa_g$. In particular, our resulting runtimes scale with $\nnz(\ma)$ instead of $\nnz(\ma_g)$, but also depend on $\ksl(\ma)$ rather than $\kappa_g$.

We remark that the other solvers we develop are also useful in robustifying against variations on the adversary in Definition~\ref{def:semirandom}. For instance, the adversary could instead aim to increase $\tau(\ma^\top \ma)$, or give additional irrelevant features (i.e.\ columns of $\ma$) such that only some subset of coordinates $x_g$ are important to recover. For brevity, we focus on the model in Definition~\ref{def:semirandom} in this work.

\subsection{Statistical linear regression}\label{ssec:regression}

The second application we give is in solving noisy variants of the linear system setting of Definition~\ref{def:semirandom}. In particular, we consider statistical regression problems with various generative models.

\begin{definition}[Statistical linear regression]\label{def:statreg}
Given full rank $\ma \in \R^{n \times d}$ and $b \in \R^d$ produced via
\begin{equation}\label{eq:genmodel}b = \ma \xtrue + \xi,\; \xi \sim \Nor(0, \msig),\end{equation}
where we wish to recover unknown $\xtrue \in \R^d$, return $x$ so that (where expectations are taken over the randomness of $\xi$) the \emph{risk} (mean-squared error) $\E[\norm{x - \xtrue}_2^2]$ is small.
\end{definition}

In this section, we define a variety of generative models (i.e.\ specifying a covariance matrix $\msig$ of the noise) for the problem in Definition~\ref{def:statreg}. For each of the generative models, applying our rescaling procedures will yield computational gains, improved risk bounds, or both. We give statistical and computational results for statistical linear regression in both the \emph{homoskedastic} and \emph{heteroskedastic} settings. In particular, when $\msig = \sigma^2 \id$ (i.e.\ the noise for every data point has the same variance), this is the well-studied homoskedastic setting pervasive in stastical modeling. When $\msig$ varies with the data $\ma$, the model is called heteroskedastic (cf.\ \cite{Greene90}). 

In most cases, we do not directly give guarantees on exact mean squared errors via our preprocessing, but rather certify (possibly loose) upper bound surrogates. We leave direct certification of conditioning and risk simultaneously without a surrogate bound as an interesting future direction.

\subsubsection{Heteroskedastic statistical guarantees}

We specify two types of heteroskedastic generative models (i.e.\ defining the covariance $\msig$ in \eqref{eq:genmodel}), and analyze the effect of rescaling a regression data matrix on reducing risk.

\paragraph{Noisy features.} Consider the setting where the covariance in \eqref{eq:genmodel} has the form $\msig = \ma \msig' \ma^\top$, for matrix $\msig' \in \PSD^d$. Under this assumption, we can rewrite \eqref{eq:genmodel} as $b = \ma(\xtrue + \xi')$, where $\xi' \sim \Nor(0, \msig')$. Intuitively, this corresponds to exact measurements through $\ma$, under noisy features $\xtrue + \xi'$. As in this case $b \in \textup{Im}(\ma)$ always, regression is equivalent to linear system solving, and thus directly solving any reweighted linear system $\mw^\half \ma x^* = \mw^\half b$ will yield $x^* = \xtrue + \xi'$. 

We thus directly obtain improved computational guarantees by computing a reweighting $\mw^\half$ with $\kappa(\ma^\top \mw \ma) = O(\ksl(\ma))$. Moreover, we note that the risk (Definition~\ref{def:statreg}) of the linear system solution $x^*$ is independent of the reweighting:
\[\E\Brack{\norm{x^* - \xtrue}_2^2} = \E\Brack{\norm{\xi'}_2^2} = \Tr\Par{\msig'}.\]
Hence, computational gains from reweighting the system are without statistical loss in the risk. 

\paragraph{Row norm noise.} Consider the setting where the covariance in \eqref{eq:genmodel} has the form
\begin{equation}\label{eq:rnnoise}\msig = \sigma^2\diag{\Brace{\norm{a_i}_2^2}_{i \in [n]}}.\end{equation}
Intuitively, this corresponds to the setting where noise is independent across examples and the size of the noise scales linearly with the squared row norm. We first recall a standard characterization of the regression minimizer.

\begin{fact}[Regression minimizer]\label{fact:regmin}
	Let the regression problem $\norm{\ma x - b}_2^2$ have minimizer $x^\star$, and suppose that $\ma^\top\ma$ is invertible. Then,
	\[x^\star = \Par{\ma^\top\ma}^{-1}\ma^\top b.\]
\end{fact}

Using Fact~\ref{fact:regmin}, we directly prove the following upper bound surrogate holds on the risk under the model \eqref{eq:genmodel}, \eqref{eq:rnnoise} for the solution to any reweighted regression problem.

\begin{lemma}\label{lem:hetriskawa}
Under the generative model \eqref{eq:genmodel}, \eqref{eq:rnnoise}, letting $\mw \in \PSD^n$ be a diagonal matrix and 
\[x^\star_w \defeq \argmin_x\Brace{\norm{\mw^\half\Par{\ma x - b}}_2^2},\]
we have
\[\E\Brack{\norm{x^\star_w - \xtrue}_2^2} \le \sigma^2 \frac{\Tr\Par{\ma^\top \mw \ma}}{\lmin\Par{\ma^\top \mw \ma}}.\]
\end{lemma}
\begin{proof}
	By applying Fact~\ref{fact:regmin}, we have that
	\[x^\star_w = \Par{\ma^\top\mw\ma}^{-1}\ma^\top\mw \Par{\ma \xtrue + \xi} = \xtrue + \Par{\ma^\top\mw\ma}^{-1}\ma^\top\mw\xi.\]
	Thus, we have the sequence of derivations
	\begin{equation}\label{eq:wrisk}
	\begin{aligned}
	\E\Brack{\norm{x^\star_w - \xtrue}_{\ma^\top\mw\ma}^2} &= \E\Brack{\norm{\Par{\ma^\top\mw\ma}^{-1}\ma^\top\mw\xi}_{\ma^\top\mw\ma}^2} \\
	&= \E\Brack{\inprod{\mw^\half\xi\xi^\top\mw^\half}{\mw^\half\ma\Par{\ma^\top\mw\ma}^{-1}\ma^\top\mw^\half}} \\
	&= \sigma^2\inprod{\diag{\Brace{w_i\norm{a_i}_2^2}}}{\mw^\half\ma\Par{\ma^\top\mw\ma}^{-1}\ma^\top\mw^\half} \\
	&\le \sigma^2\Tr\Par{\ma^\top\mw\ma}.
	\end{aligned}
	\end{equation}
	The last inequality used the $\ell_1$-$\ell_\infty$ matrix H\"older inequality and that $\mw^\half\ma\Par{\ma^\top\mw\ma}^{-1}\ma^\top\mw^\half$ is a projection matrix, so $\|\mw^\half\ma\Par{\ma^\top\mw\ma}^{-1}\ma^\top\mw^\half\|_\infty = 1$. Lower bounding the squared $\ma^\top \mw \ma$ norm by a $\lmin(\ma^\top \mw \ma)$ multiple of the squared Euclidean norm yields the conclusion.
\end{proof}

We remark that the analysis in Lemma~\ref{lem:hetriskawa} of the surrogate upper bound we provide was loose in two places: the application of H\"older and the norm conversion. Lemma~\ref{lem:hetriskawa} shows that the risk under the generative model \eqref{eq:rnnoise} can be upper bounded by a quantity proportional to $\tau(\ma^\top \mw \ma)$, the average conditioning of the reweighted matrix. 

Directly applying Lemma~\ref{lem:hetriskawa} or further using the inequality $\tau(\ma^\top \mw \ma) \le d \kappa(\ma^\top \mw \ma)$ \eqref{eq:taubounds}, our risk upper bounds improve with the conditioning or average conditioning of the reweighted system. Hence, our rescaling procedures improve both the computational and statistical guarantees of regression under this generative model, albeit only helping the latter through an upper bound.

\subsubsection{Homoskedastic statistical guarantees}

In this section, we work under the homoskedastic generative model assumption. In particular, throughout the covariance matrix in \eqref{eq:genmodel} will be a multiple of the identity:
\begin{equation}\label{eq:homonoise}
\msig = \sigma^2 \id.
\end{equation}
We begin by providing a risk upper bound under the model \eqref{eq:genmodel}, \eqref{eq:homonoise}.
\begin{lemma}\label{lem:homrisk}
	Under the generative model \eqref{eq:genmodel}, \eqref{eq:homonoise}, let $x^\star \defeq \argmin_x\{\norm{\ma x - b}_2^2\}$. Then,
	\begin{equation}\label{eq:homrisk}\E\Brack{\norm{x^\star - \xtrue}_{\ma^\top\ma}^2} = \sigma^2 d \implies \E\Brack{\norm{x^* - \xtrue}_{2}^2} \le \frac{\sigma^2 d}{\lmin(\ma^\top\ma)}.\end{equation}
\end{lemma}
\begin{proof}
	Using Fact~\ref{fact:regmin}, we compute
	\begin{align*}x^\star - \xtrue &= \Par{\ma^\top\ma}^{-1}\ma^\top b - \xtrue\\
	&= \Par{\ma^\top\ma}^{-1}\ma^\top \Par{\ma \xtrue + \xi} - \xtrue = \Par{\ma^\top\ma}^{-1}\ma^\top\xi.\end{align*}
	Therefore via directly expanding, and using linearity of expectation,
	\begin{align*}\E\Brack{\norm{x^* - \xtrue}_{\ma^\top\ma}^2} &= \E\Brack{\norm{\ma\Par{\ma^\top\ma}^{-1}\ma^\top\xi}_2^2} \\
	&= \E\Brack{\inprod{\xi\xi^\top}{\ma\Par{\ma^\top\ma}^{-1}\ma^\top}} = \sigma^2 \Par{\ma\Par{\ma^\top\ma}^{-1}\ma^\top} = \sigma^2 d.
	\end{align*}
	The final implication follows from $\lmin(\ma^\top\ma)\norm{x^* - \xtrue}_2^2 \le \norm{x^* - \xtrue}_{\ma^\top\ma}^2$.
\end{proof}

Lemma~\ref{lem:homrisk} shows that in regards to our upper bound (which is loose in the norm conversion at the end), the notion of adversarial semi-random noise is at odds in the computational and statistical senses. Namely, given additional rows of the matrix $\ma$, the bound \eqref{eq:homrisk} can only improve, since $\lmin$ is monotonically increasing as rows are added. To address this, we give guarantees about recovering reweightings which match the best possible upper bound anywhere along the ``computational-statistical tradeoff curve.'' We begin by providing a weighted analog of Lemma~\ref{lem:homrisk}.

\begin{lemma}\label{lem:weighthomrisk}
Under the generative model \eqref{eq:genmodel}, \eqref{eq:homonoise}, letting $\mw \in \PSD^n$ be a diagonal matrix and 
\[x^\star_w \defeq \argmin_x\Brace{\norm{\mw^\half\Par{\ma x - b}}_2^2},\]
we have
\begin{equation}\label{eq:homriskweight}\E\Brack{\norm{x^\star_w - \xtrue}_{2}^2} \le \sigma^2 d \cdot \frac{\norm{w}_\infty}{\lmin\Par{\ma^\top \mw \ma}}.\end{equation}
\end{lemma}
\begin{proof}
By following the derivations \eqref{eq:wrisk} (and recalling the definition of $x^\star_w$),
\begin{equation}\label{eq:bigtracebound}
\begin{aligned}
\E\Brack{\norm{x^\star_w - \xtrue}_{\ma^\top\mw\ma}^2} &= \E\Brack{\inprod{\xi\xi^\top}{\mw\ma\Par{\ma^\top\mw\ma}^{-1}\ma^\top\mw}} \\
&= \sigma^2\Tr\Par{\mw\ma\Par{\ma^\top\mw\ma}^{-1}\ma^\top\mw}.
\end{aligned}
\end{equation}
Furthermore, by $\mw \preceq \norm{w}_\infty \id$ we have $\ma^\top \mw^2 \ma \preceq \norm{w}_\infty \ma^\top \mw \ma$. Thus,
\[\Tr\Par{\mw\ma\Par{\ma^\top\mw\ma}^{-1}\ma^\top\mw} = \inprod{\ma^\top \mw^2 \ma}{\Par{\ma^\top\mw\ma}^{-1}} \le \norm{w}_\infty \Tr(\id)= d\norm{w}_\infty.\]
Using this bound in \eqref{eq:bigtracebound} and converting to Euclidean norm risk yields the conclusion.
\end{proof}

Lemma~\ref{lem:weighthomrisk} gives a quantitative version of a computational-statistical tradeoff curve. Specifically, we give guarantees which target the best possible condition number of a $0$-$1$ reweighting, subject to a given level of $\lmin(\ma^\top \mw \ma)$. In the following discussion we assume there exists $\ma_g \subseteq \ma$, a subset of rows, satisfying (for known $\kappa_g$, $\nu_g$, and sufficiently small constant $\eps \in (0, 1)$)
\begin{equation}\label{eq:kappanudef} \kappa_g \le \kappa\Par{\ma_g^\top \ma_g} \le (1 + \eps)\kappa_g,\; \frac{1}{\lmin\Par{\ma_g^\top \ma_g}} \le \nu_g.\end{equation}
Our key observation is that we can use existence of a row subset satisfying \eqref{eq:kappanudef}, combined with a slight modification of Algorithm~\ref{alg:decidempc}, to find a reweighting $w$ such that
\begin{equation}\label{eq:boundedreweighting}
\kappa\Par{\ma^\top \mw \ma} = O\Par{\kappa_g},\; \frac{\norm{w}_\infty}{\lmin\Par{\ma^\top \mw \ma}} = O(\nu_g).
\end{equation}

\begin{lemma}
Consider running Algorithm~\ref{alg:decidempc}, with the modification that in Line 7, we set
\begin{equation}\label{eq:tmadef}
\begin{gathered}
x_t \gets \text{an } \frac{\eps}{10} \text{-multiplicative approximation of } \argmax_{\substack{\sum_{i \in [n]} w_i \tma_i \preceq \id \\ x \in \R^n_{\ge 0}}} \inprod{\kappa v_t}{w}, \\
\text{where for all } i \in [n],\; \tma_i \defeq \begin{pmatrix} \ma_i & \mzero_{d \times n} \\ \mzero_{n \times d} & \diag{\frac{\kappa_g}{\nu_g} e_i} \end{pmatrix}.
\end{gathered}
\end{equation}
Then, if \eqref{eq:kappanudef} is satisfied for some $\ma \in \R^{n \times d}$ and row subset $\ma_g \subseteq \ma$, Algorithm~\ref{alg:decidempc} run on $\kappa \gets \kappa_g$ and $\{\ma_i = a_ia_i^\top\}_{i \in [n]}$ where $\{a_i\}_{i \in [n]}$ are rows of $\ma$ will produce $w$ satisfying \eqref{eq:boundedreweighting}.
\end{lemma}
\begin{proof}
We note that each matrix $\tma_i$ is the same as the corresponding $\ma_i$, with a single nonzero coordinate along the diagonal bottom-right block. The proof is almost identical to the proof of Lemma~\ref{lem:decidempccorrect}, so we highlight the main differences here. The main property that Lemma~\ref{lem:decidempccorrect} used was that Line 9 did not pass, which lets us conclude \eqref{eq:ygbound}. Hence, by the approximation guarantee on each $x_t$, it suffices to show that for any $\my_t \in \PSD^d$ with $\Tr(\my_t) = 1$, (analogously to \eqref{eq:optrelate}),
\begin{equation}\label{eq:optgoodtma}
\max_{\substack{\sum_{i \in [n]} w_i \tma_i \preceq \id \\ x \in \R^n_{\ge 0}}} \kappa_g \inprod{\my_t}{\sum_{i \in [n]} w_i \ma_i} \ge 1 - O(\eps).
\end{equation}
However, by taking $w$ to be the $0$-$1$ indicator of the rows of $\ma_g$ scaled down by $\lmax(\ma_g^\top \ma_g)$, we have by the promise \eqref{eq:kappanudef} that
\begin{equation}\label{eq:wtmagood}
\sum_{i \in [n]} w_i \tma_i = \frac{1}{\lam_{\max}(\ma_g^\top \ma_g)}\preceq \id \impliedby \frac{1}{\lam_{\max}(\ma_g^\top \ma_g)} \ma_g^\top \ma_g \preceq \id,\; \frac{\kappa_g}{\nu_g} \cdot \frac{1}{\lam_{\max}(\ma_g^\top \ma_g)} \le 1.
\end{equation}
Now, it suffices to observe that \eqref{eq:wtmagood} implies our indicator $w$ is feasible for \eqref{eq:optgoodtma}, so
\[\max_{\substack{\sum_{i \in [n]} w_i \tma_i \preceq \id \\ x \in \R^n_{\ge 0}}} \kappa_g \inprod{\my_t}{\sum_{i \in [n]} w_i \ma_i} \ge \frac{\lmin\Par{\ma_g^\top \ma_g}}{\lmax\Par{\ma_g^\top \ma_g}} \cdot \kappa_g \ge 1 - O(\eps).\]
The remainder of the proof is identical to Lemma~\ref{lem:decidempccorrect}, where we note the output $w$ satisfies
\[\sum_{i \in [n]} w_i \tma_i \preceq \id,\; \sum_{i \in [n]} w_i \ma_i \succeq \frac{1 - O(\eps)}{\kappa_g} \id,\]
which upon rearrangement and adjusting $\eps$ by a constant yields \eqref{eq:boundedreweighting}.
\end{proof}

By running the modification of Algorithm~\ref{alg:decidempc} described for a given level of $\nu_g$, it is straightforward to perform an incremental search on $\kappa_g$ to find a value satisfying the bound \eqref{eq:boundedreweighting} as described in Theorem~\ref{thm:leftslow}. It is simple to verify that the modification in \eqref{eq:tmadef} is not the dominant runtime in any of Theorems~\ref{thm:leftslow},~\ref{thm:rightslow}, or~\ref{thm:slowsqrt} since the added constraint is diagonal and $\tma_i$ is separable. Hence, for every ``level'' of $\nu_g$ in \eqref{eq:kappanudef} yielding an appropriate risk bound \eqref{eq:homriskweight}, we can match this risk bound up to a constant factor while obtaining computational speedups scaling with $\kappa_g$.  	
	\subsection*{Acknowledgments}
	
	AS was supported in part by a Microsoft Research Faculty Fellowship, NSF CAREER Award CCF-1844855, NSF Grant CCF-1955039, a PayPal research award, and a Sloan Research Fellowship. KT was supported by a Google Ph.D.\ Fellowship, a Simons-Berkeley VMware Research Fellowship, a Microsoft Research Faculty Fellowship, NSF CAREER Award CCF-1844855, NSF Grant CCF-1955039, and a PayPal research award. We would like to thank Huishuai Zhang for his contributions to an earlier version of this project.
	
	\bibliographystyle{alpha}	
	\bibliography{diagonal-scaling}
	\newpage
	\begin{appendix}

\section{Discussion of (specialized) mixed packing-covering SDP formulations}
\label{app:mpcdiscuss}

In this section, we give a brief discussion of the generality of different specialized mixed packing-covering SDP formulations, as stated in \cite{JambulapatiLLPT20}, \cite{JambulapatiSS18}, and this paper. We also discuss the application of our algorithm in Section~\ref{sec:slowoptimal} to the formulation in \cite{JambulapatiSS18}.

In full generality, the mixed packing-covering SDP problem is parameterized by matrices \[\{\mpack_i\}_{i \in [n]}, \mpack, \{\mcov_i\}_{i \in [n]}, \mcov \in \PSD^d,\]
and asks to find the smallest $\mu > 0$ such that there exists $w \in \R^n_{\ge 0}$ with
\begin{equation}\label{eq:mpcmostgen}\sum_{i \in [n]} w_i \mpack_i \preceq \mu \mpack,\; \sum_{i \in [n]} w_i \mcov_i \succeq \mcov.\end{equation}
By redefining $\mpack_i \gets \frac 1 \mu \mpack^{-\half} \mpack_i \mpack^{-\half}$ and $\mcov_i \gets \mcov^{-\half} \mcov_i \mcov^{-\half}$ for all $i \in [n]$ and a given $\mu > 0$, the optimization problem in \eqref{eq:mpcmostgen} is equivalent to testing whether there exists $w \in \R^n_{\ge 0}$ such that 
\begin{equation}\label{eq:mpcid}\sum_{i \in [n]} w_i \mpack_i \preceq \sum_{i \in [n]} w_i \mcov_i.\end{equation}
The above formulation \eqref{eq:mpcid} is studied by \cite{JambulapatiLLPT20}, and no ``width-independent'' solver is known in the literature (namely, testing whether \eqref{eq:mpcid} is feasible to a multiplicative $1 + \eps$ factor with an iteration count polynomial in $\eps^{-1}$ and polylogarithmic in other problem parameters).

This work and \cite{JambulapatiSS18} develop different algorithms for solving specializations of \eqref{eq:mpcmostgen}, \eqref{eq:mpcid} where the packing and covering matrices $\{\mpack_i\}_{i \in [n]}, \{\mcov_i\}_{i \in [n]}$, as well as the constraints $\mpack$, $\mcov$, are multiples of each other. In particular, Problem 3.1 of \cite{JambulapatiSS18} asks, given matrices $\{\ma_i\}_{i \in [n]}, \mb \in \PSD^d$ and scalar $\kappa > 1$, to test feasibility of finding $w \in \R^n_{\ge 0}$ such that
\begin{equation}\label{eq:jssmpc}\mb \preceq \sum_{i \in [n]} w_i \ma_i \preceq \kappa \mb.\end{equation}
We note that \eqref{eq:jssmpc} is a specialization of the feasibility form of \eqref{eq:mpcmostgen}. Furthermore, the formulation \eqref{eq:jssmpc} captures our scaling problems; by setting $\ma_i = a_ia_i^\top$ and $\mb = \id$, and letting $\ma \in \R^{n \times d}$ have rows $\{a_i\}_{i \in [n]}$, \eqref{eq:jssmpc} recovers our inner scaling problem (determining $\ksl(\ma)$). On the other hand, by setting $\ma_i = e_ie_i^\top$ for all $i \in [n]$ where $n = d$, and letting $\mb = \mk \in \PD^d$ be an arbitrary positive definite matrix, \eqref{eq:jssmpc} recovers our outer scaling problem (determining $\ksk(\mk)$). In the case of outer scaling, this formulation constitutes a different strategy for obtaining a near-optimal reweighting than our ``imagining a basis'' strategy in Section~\ref{ssec:ksrslow}.

In \cite{JambulapatiSS18}, an algorithm was given for determining the optimal $\kappa$ in \eqref{eq:jssmpc} up to a $1 + \eps$ multiplicative factor, with a runtime polynomial in $\kappa$ and $\eps^{-1}$ (and polylogarithmic in other parameters). The algorithm of \cite{JambulapatiSS18} was based on generalizing techniques from \cite{Lee017}, and assumed efficient access to inverses and factorizations of weighted combinations $\sum_{i \in [n]} w_i \ma_i$. The \cite{JambulapatiSS18} algorithm also applies to our scaling settings, albeit at larger polynomial dependences on $\kappa$.

 For the specializations of \eqref{eq:jssmpc} described above which capture our inner and outer scaling problems, we provide an alternative algorithm framework in Sections~\ref{sec:slowoptimal} and~\ref{sec:kernel} which refines the dependence of the \cite{JambulapatiSS18} algorithm on the optimal $\kappa$, scaling as $\kappa^{1.5}$. We believe our algorithm has implications for the more general problem of \eqref{eq:jssmpc} as well, under the same assumptions as \cite{JambulapatiSS18} (namely obtaining a rate scaling as $\kappa^{1.5}$ assuming appropriate efficient access to $\sum_{i \in [n]} w_i \ma_i$). This would sharpen the polynomial dependence on $\kappa$ in Theorem 3.3 of \cite{JambulapatiSS18}. However, the implementation would be somewhat different compared to the techniques used in Sections~\ref{sec:slowoptimal} and~\ref{sec:kernel}, as we would require rational approximations to the square root to efficiently simulate access to $\mb^{-\half}$, as opposed to the polynomials used in Section~\ref{sec:kernel}. For brevity and to not detract from the focus of this work, we choose to not pursue these extensions in this paper.

\section{Deferred proofs from Sections~\ref{sec:prelims} and~\ref{sec:diagonalone}}
\label{app:prelims}

\restatepolysqrt*
\begin{proof}	
	We will instead prove the following fact: for any $\eps \in (0, 1)$, there is an explicit degree-$O\left(\sqrt{\kappa}\log \frac \kappa \eps\right)$ polynomial $p$ satisfying
	\begin{align*}
	\max_{x\in [\frac 1 \kappa,1]} |p(x) - \sqrt{x}| \leq \epsilon.
	\end{align*}	
	The conclusion for arbitrary scalars with multiplicative range $[\mu, \kappa\mu]$ will then follow from setting $\eps = \delta \kappa^{-\half}$ (giving a multiplicative error guarantee), and the fact that rescaling the range $[\frac 1 \kappa, 1]$ will preserve this multiplicative guarantee (adjusting the coefficients of the polynomial as necessary, since $\mu$ is known). Finally, the conclusion for matrices follows since $p(\mm)$ and $\mm^\half$ commute.

	Denote $\gamma = \frac 1 \kappa$ for convenience. We first shift and scale the function $\sqrt{x}$ to adjust the region of approximation from $[\gamma,1]$ to $[-1,1]$. In particular, let $h(x) = \sqrt{\frac{1-\gamma}{2}x + \frac{\gamma+1}{2}}$. If we can find some degree-$\Delta$ polynomial $g(x)$ with $|g(x)-h(x)|\leq \epsilon$ for all $x\in [-1,1]$, then 
	\[p(x) = g\left(\frac{2}{1-\gamma}x\ -\ \frac{1 + \gamma}{1-\gamma}\right)\]
	provides the required approximation to $\sqrt{x}$. 
	
	To construct $g$, we take the Chebyshev interpolant of $h(x)$ on the interval $[-1,1]$. Since $h$ is analytic on $[-1,1]$, we can apply standard results on the approximation of analytic functions by polynomials, and specifically Chebyshev interpolants. Specifically, by Theorem 8.2 in \cite{Trefethen:2012}, if $h(z)$ is analytic in an open Bernstein ellipse with parameter $\rho$ in the complex plane, then:
	\begin{align*}
	\max_{x\in [-1,1]} |g(x) - h(x)| \leq \frac{4M}{\rho-1}\rho^{-\Delta}, 
	\end{align*}
	where $M$ is the maximum of $|h(z)|$ for $z$ in the ellipse. 
	It can be checked that $h(x)$ is analytic on an open Bernstein ellipse with parameter $\rho = \frac{1+\sqrt{\gamma}}{1-\sqrt{\gamma}}$ --- i.e.\ with major axis length $\rho + \rho^{-1} = 2\frac{1+\gamma}{1-\gamma}$. We can then check that $M = \sqrt{1+\gamma} \leq \sqrt{2}$ and $\rho - 1 \geq 2\sqrt{\gamma}$. Since for all $\gamma < 1$, 
	\[\left(\frac{1-\sqrt{\gamma}}{1+\sqrt{\gamma}}\right)^{1/2\gamma} \leq \frac{1}{e},\]
	we conclude that $\frac{4M}{\rho-1}\rho^{-\Delta} \leq \epsilon$ as long as $\Delta \geq \frac{1}{2\gamma}\log\left(\frac{\epsilon}{\sqrt{2\gamma}}\right)$, which completes the proof. 
\end{proof}

\restateddd*
\begin{proof}
	Throughout let $\ksk \defeq \ksk(\mk)$ for notational convenience. Let $\mw_\star$ obtain the minimum in the definition of $\ksk$ (Definition~\ref{def:optimal_condition_number}) and let $\mb = \mw_\star^\half\mk\mw_\star^\half$. Also let $\mw_\mb$ be the inverse of a diagonal matrix with the same entries as $\mb$'s diagonal. Note that $\kappa(\mb) = \ksk $ and $\mw_\mb^{\half} \mb \mw_\mb^{\half} = \mw^{\half} \mk \mw^{\half}$. So, to prove Proposition \ref{prop:diagdimensiondependence}, it suffices to prove that
	\[
	\kappa\Par{\mw_\mb^{\half} \mb \mw_\mb^{\half}} \leq \min\left(m,\sqrt{\nnz(\mk)}\right)\cdot \ksk.
	\]
	Let $d_{\max}$ denote the largest entry in $\mw_\mb^{-1}$. We have that $d_{\max} \leq \lmax(\mb)$. Then let $\mm = (d_{\max}\mw_\mb)^{\half} \mb (d_{\max}\mw_\mb)^{\half}$ and note that all of $\mm$'s diagonal entries are equal to $d_{\max}$ and $\kappa\Par{\mm} = 	\kappa(\mw_\mb^{\half} \mb \mw_\mb^{\half})$. Moreover, since $d_{\max}\mw_\mb$ has all entries $\geq 1$, $\lmin(\mm) \geq \lmin(\mb)$. Additionally, since a PSD matrix must have its largest entry on the diagonal, we have that $\|\mm\|_{\text{F}}^2 \leq \nnz(\mm)d_{\max}^2 \leq \nnz(\mm)\lmax(\mb)^2$. Accordingly, $\lmax(\mm) = \|\mm\|_2 \leq \|\mm\|_{\text{F}} \leq \sqrt{\nnz(\mm)}\lmax(\mb).$
	
	From this lower bound on $\lmin(\mm)$ and upper bound on $\lmax(\mm)$, we have that
	\begin{align*}
	\kappa\Par{\mw^{\half} \mk \mw^{\half}} = \kappa\Par{\mm} \leq \frac{\sqrt{\nnz(\mm)}\lmax(\mb)}{\lmin(\mb)} = \sqrt{\nnz(\mm)}\cdot\kappa(\mb).
	\end{align*}
	
	This proves one part of the minimum in Proposition \ref{prop:diagdimensiondependence}. The second, which was already proven in \cite{Sluis:1969} follows similarly. In particular, by the Gershgorin circle theorem we have that $\lmax(\mm) \leq \max_{i\in [d]}\|\mm_{i:}\|_1$, where $\mm_{i:}$ denotes the $i^\text{th}$ row for $\mm$.  
	Since all entries in $\mm$ are bounded by $d_{\max} \leq \lmax(\mb)$, we have that $\max_{i\in [d]}\|\mm_{i:}\|_1 \leq m\lmax(\mb)$, and thus
	\begin{align*}
	\kappa\Par{\mw^{\half} \mk \mw^{\half}} &= \kappa\Par{\mm} \leq \frac{m\lmax(\mb)}{\lmin(\mb)} = m\cdot\kappa(\mb). \qedhere
	\end{align*}
\end{proof}

\restateprodkappa*
\begin{proof}
It is straightforward from $\lam_{\min}(\ma) \id \preceq \ma \preceq \lam_{\max}(\ma)\id$ that 
\[\sqrt{\lam_{\min}(\ma)} \norm{u}_2 \le \norm{\ma^\half u}_2 \le \sqrt{\lam_{\max}(\ma)} \norm{u}_2,\]
and an analogous fact holds for $\mb$. 
Hence, we can bound the eigenvalues of $\ma^\half \mb \ma^\half$:
\begin{gather*}
\lmax\Par{\ma^\half\mb\ma^\half} = \max_{\norm{u}_2 = 1} u^\top \ma^\half \mb \ma^\half \mb u \le \lam_{\max}(\ma)\max_{\norm{v}_2 = 1} v^\top \mb v = \lam_{\max}(\ma) \lam_{\max}(\mb), \\
\lmin\Par{\ma^\half \mb \ma^\half} = \min_{\norm{u}_2 = 1} u^\top \ma^\half \mb \ma^\half \mb u \ge \lam_{\min}(\ma)\min_{\norm{v}_2 = 1} v^\top \mb v = \lam_{\min}(\ma)\lam_{\min}(\mb).
\end{gather*}
Dividing the above two equations yields the claim.
\end{proof} 	%

\section{Deferred proofs from Section~\ref{sec:slowoptimal}}
\label{app:slowoptimal}

We give a proof of Proposition~\ref{prop:optv} in this section. First, we recall an algorithm for the testing variant of a pure packing SDP problem given in \cite{Jambulapati0T20}.

\begin{proposition}[Theorem 5, \cite{Jambulapati0T20}]\label{prop:decisionpack}
	There is an algorithm, $\algtest$, which given matrices $\{\ma_i\}_{i \in [n]}$ and a parameter $C$, is an $\eps$-tolerant tester for the decision problem
	\begin{equation}\label{eq:packdecision}\text{does there exist } w \in \Delta^n \text{ such that } \sum_{i \in [n]} w_i \ma_i \preceq C\id?\end{equation}
	The algorithm $\algtest$ succeeds with probability $\ge 1 - \delta$ and runs in time 
	\[O\Par{\tmv\Par{\Brace{\ma_i}_{i \in [n]}} \cdot \frac{\log^2(nd(\delta\eps)^{-1})\log^2 d}{\eps^5}}.\]
\end{proposition}

\begin{proof}[Proof of Proposition~\ref{prop:optv}]
As an immediate result of Proposition~\ref{prop:decisionpack}, we can solve \eqref{eq:optv} to multiplicative accuracy $\eps$ using a binary search. This reduction is derived as Lemma A.1 of \cite{JambulapatiLLPT20}, but we give a brief summary here. We subdivide the range $[\opt_-, \opt_+]$ into $K$ buckets of multiplicative range $1 + \frac \eps 3$, i.e.\ with endpoints $\opt_- \cdot (1 + \frac \eps 3)^k$ for $0 \le k \le K$ and
\[K = O\Par{\frac 1 \eps \cdot \log\Par{\frac{\opt_+}{\opt_-}}}.\]
We then binary search over $0 \le k \le K$ to determine the value of $\opt(v)$ to $\eps$-multiplicative accuracy, returning the largest endpoint for which the decision variant in Proposition~\ref{prop:decisionpack} returns feasible (with accuracy $\frac \eps 3$). By the guarantees of Proposition~\ref{prop:decisionpack}, the feasible point returned by Proposition~\ref{prop:decisionpack} for this endpoint will attain an $\eps$-multiplicative approximation to the optimization variant \eqref{eq:optv}, and the runtime is that of Proposition~\ref{prop:decisionpack} with an overhead of $O(\log K)$.
\end{proof} 	%

\section{Deferred proofs from Section~\ref{sec:kernel}}
\label{app:kernel}

\restateopnormprod*
\begin{proof}
	We begin with the first entry in the above minimum. Let $v$ be the unit vector with $\norm{\mb v}_2 = \norm{\mb}_\infty$, and note $\norm{\ma \mb v}_2 \ge \frac{1}{\kappa(\ma)} \norm{\ma}_\infty\norm{\mb v}_2$ by definition of $\kappa(\ma)$. Hence,
	\[\norm{\ma \mb}_\infty \ge \norm{\ma \mb v}_2 \ge \frac{1}{\kappa(\ma)} \norm{\ma}_\infty\norm{\mb v}_2 = \frac{1}{\kappa(\ma)} \norm{\ma}_\infty\norm{\mb}_\infty.\]
	We move onto the second entry. Let $v$ be a vector such that $\norm{\ma v}_2$ and $\norm{\mb \ma v}_2 = \norm{\mb}_\infty$; note that $\norm{v}_2 \le \frac{\kappa(\ma)}{\norm{\ma}_\infty}$. The conclusion follows from rearranging the following display:
	\[\frac{\kappa(\ma)\norm{\mb\ma}_\infty}{\norm{\ma}_\infty} \ge \norm{\mb\ma}_\infty \norm{v}_2 \ge \norm{\mb\ma v}_2 = \norm{\mb}_\infty.\]
\end{proof}

\restatelambounds*
\begin{proof}
	To see the first claim, the largest eigenvalue of $\mk + \frac 1 \eps \lmax(\mk)\id$ is at most $(1 + \frac 1 \eps)\lmax(\mk)$ and the smallest is at least $\frac 1 \eps \lmax(\mk)$, so the condition number is at most $1 + \eps$ as desired.
	
	To see the second claim, it follows from the fact that outer rescalings preserve Loewner order, and then combining
	\begin{align*}
	\mk \preceq \mk + \lam \id \implies \lmax\Par{\mw^\half\mk \mw^\half} \le \lmax\Par{\mw^\half \Par{\mk + \lam\id} \mw^\half}, \\
	\mk \succeq \frac 1 {1 + \eps}\Par{\mk + \lam \id} \implies \lmin\Par{\mw^\half \mk \mw^\half} \ge \frac 1 {1 + \eps}\lmin\Par{\mw^\half\Par{\mk + \lam \id}\mw^\half}.
	\end{align*}
\end{proof}

\restatebetweenphase*
\begin{proof}
	First, because outer rescalings preserve Loewner order, it is immediate that
	\[\mk + \frac \lam 2 \id \preceq \mk + \lam \id \implies \lmax\Par{\mw^\half\Par{\mk + \frac \lam 2}\mw^\half \id} \le \lmax\Par{\mw^\half\Par{\mk + \lam \id}\mw^\half} .\]
	Moreover, the same argument shows that
	\[\half \mk + \frac \lam 2 \id \preceq \mk + \frac \lam 2 \id \implies \lmin\Par{\mw^\half\Par{\mk + \frac \lam 2 \id}\mw^\half} \ge \half \lmin\Par{\mw^\half \Par{\mk + \id} \mw^\half}.\]
	Combining the above two displays yields the conclusion.
\end{proof} 	\end{appendix}

\end{document}